\documentclass[journal]{IEEEtran}
                                
\ifCLASSOPTIONcompsoc      
  \usepackage[caption=false,font=normalsize,labelfont=sf,textfont=sf]{subfig}
\else
  \usepackage[caption=false,font=footnotesize]{subfig}
\fi  
           
\usepackage{cite}
\usepackage{amssymb}
\usepackage{mathrsfs}
\usepackage{amsmath}
\usepackage{graphicx}
\usepackage{array}
\usepackage{multirow}
\usepackage{color} 
\usepackage{xcolor}
\usepackage{epstopdf}
\usepackage{cite}
\usepackage{amsfonts} 
  
\usepackage[twocolumn,letterpaper]{geometry}

\newtheorem{theorem}{\bf{Theorem}}

\newtheorem{condition}{\bf{Assumption}}

\newtheorem{definition}{\bf{Definition}}

\newtheorem{lemma}{\bf{Lemma}}

\newtheorem{problem}{\bf{Problem}}

\newtheorem{remark}{\bf{Remark}}

\begin{document}
 
\title{\Large Attack-Resilient Distributed Convex Optimization of Linear Multi-Agent Systems Against Malicious Cyber-Attacks over Random Digraphs}
 
\author
{Zhi~Feng and~Guoqiang~Hu 

\thanks{
This work was supported by Singapore Ministry of Education Academic Research Fund Tier 1 RG180/17(2017-T1-002-158) and in part by the Wallenberg-NTU Presidential Postdoctoral Fellow Start-Up Grant. Z. Feng and G. Hu are with the School of Electrical and Electronic Engineering, Nanyang Technological University, Singapore 639798 (E-mail: zhifeng@ntu.edu.sg; gqhu@ntu.edu.sg). 
}  
} 
 
 
\maketitle 

\begin{abstract}
This paper addresses a resilient exponential distributed convex optimization problem for a heterogeneous linear multi-agent system under Denial-of-Service (DoS) attacks over random digraphs. The random digraphs are caused by unreliable networks and the DoS attacks, allowed to occur aperiodically, refer to an interruption of the communication channels carried out by the intelligent adversaries. In contrast to many existing distributed convex optimization works over a prefect communication network, the global optimal solution might not be sought under the adverse influences that result in performance degradations or even failures of optimization algorithms. The aforementioned setting poses certain technical challenges to optimization algorithm design and exponential convergence analysis. In this work, several resilient algorithms are presented such that a team of agents minimizes a sum of local non-quadratic cost functions in a safe and reliable manner with global exponential convergence. Inspired by the preliminary works in \cite{Hu15IFAC,Hu15Cyber,Hu15IJNRC,Hu19TCST}, an explicit analysis of frequency and duration of attacks is investigated to guarantee exponential optimal solutions. 
Numerical simulation results are presented to demonstrate the effectiveness of the proposed design. 
\end{abstract}
 
\vspace*{-2pt} 

\IEEEpeerreviewmaketitle

\vspace*{-10pt}
\section{Introduction}
\vspace*{-2pt}
Distributed convex optimization over a multi-agent system has attracted growing attention over the last decade due to its potential applications involving source localization in sensor networks \cite{Rabbat07Conf}, resource allocation in multi-cell networks \cite{Shen12TSP}, energy and thermal comfort optimization in smart building \cite{Gupta15TSCI},  economic dispatch in smart grid \cite{Bai19TCST}, etc. The gradient-based distributed algorithms have been widely provided in the existing works, which build on either continuous-time or discrete-time agent dynamics to seek an optimal solution \cite{Schenato16TAC,Gharesifard14TAC,Hu16TAC,Hu19TCNS,Zhao17TAC,Kia15AT,Liu15TAC,LiuTCNS18,Ding20TAC}. Each agent needs to calculate a global optimizer based on the information exchange over a communication network. The network security plays a fundamental yet important role in information transmission. Unfortunately, due to malicious attacks, (e.g., DoS attacks, deception attacks (false data injections, replay attacks), disclosure attacks (eavesdropping), and Byzantine attacks (faulty agents)), the secure network environment is hardly guaranteed in practice. The malicious attacks interrupt, incorrect, or tamper transmitting information so that efficiency of distributed optimization algorithm is degraded significantly or even failed. In light of wide applications of distributed optimization algorithms in cyber-physical systems (safety-critical), and inspired by studies of security issues in consensus works (e.g., \cite{Mo14SP,Hu15IFAC,Hu15Cyber,Hu15IJNRC,Hu19TCST,Feng17AT,Chen19SMC}), it is highly desirable to determine how resilient distributed optimization algorithms are against cyber-attacks. Motivated by the observation,
we aim to address resilient problems of distributed optimization with certain resilience against DoS attacks over unreliable networks. 

\subsection{Related Literature Review} 
\vspace*{-1pt}
\textit{Continuous-Time Distributed Convex Optimization}: paralleled with the discrete-time optimization schemes, the continuous-time distributed optimization algorithms have attracted much attention due to the well-developed continuous-time analysis techniques in control and the real-world cyber-physical system. In particular, the Zero-gradient-sum \cite{Hu19TCNS} and Newton-Raphson \cite{Schenato16TAC} algorithms are designed based on the positive and bounded Hessian of the local cost functions, while the Lagrangian-based algorithm based on the local gradients is adopted in \cite{Gharesifard14TAC}. The penalty-based optimization strategies are developed in  \cite{Hu19TCNS} and \cite{Hu16TAC}, while the works  investigate the multi-agent system with single-integrator dynamics only. The authors in \cite{Liu15TAC} further consider distributed optimization of second-order multi-agent systems. Adaptive schemes are designed in \cite{Zhao17TAC} to achieve distributed optimization of linear agents via nonsmooth signum functions and gradients that satisfy special structures. All above designs require continuous communication. To remove this requirement, the time-based and event-triggered based strategies are presented in \cite{Kia15AT,LiuTCNS18,Ding20TAC} over undirected and connected graphs. Besides, these algorithms need known initial states of each agent, which is difficult to verify in practice.

\vspace{2.5pt}
\textit{Resilient Distributed Convex Optimization}: the closest related works on this topic have been recently published in \cite{Sundaram15Conf,Su16ACC,Sundaram19TAC,Zhao20TAC} in which the optimal solutions are obtained for first-order discrete-time multi-agent systems under Byzantine attacks (faulty nodes). In particular, the authors in \cite{Sundaram15Conf} present the resilient optimization algorithm by removing  $ F $ (maximum amount of tolerable faults) nodes' largest and smallest states at each iteration, such that the optimal solution converges to a convex hull of the set of all non-faulty nodes' local minimizers. This algorithm is adopted in \cite{Su16ACC} to deal with constrained optimization problems. The local filtering consensus-based algorithms against $ F $ faulty nodes are developed in \cite{Sundaram19TAC}, where the optimal solution is achieved under the $r$-robust condition. The modified algorithm named RDO-T is developed in \cite{Zhao20TAC}, where the number of faulty agents is allowed to be any large. These aforementioned works assume that the faulty nodes and the removal of $ F $ states are known a prior to the designer. Moreover,
these algorithms rely on the network connectivity of the complete graph or the undirected and connected graph.    
 
\vspace*{-8pt}
\subsection{Main Contributions} 
\vspace*{-1pt}
This paper is concerned with a resilient study of distributed optimization algorithms against cyber-attacks over random digraphs. The malicious DoS attacks, allowed to occur aperiodically, aim to interrupt information transmission. In addition to DoS attacks, the random digraphs are induced by unreliable networked constraints. The main contributions of this paper are summarized as follows. (1) inspired by our designs in \cite{Hu15IFAC,Hu15Cyber,Hu15IJNRC,Hu19TCST}, this is the first to present time-based and event-based distributed optimization algorithms to achieve resilient distributed optimization of heterogeneous linear multi-agent networks against DoS attacks over random digraphs. The proposed algorithms that rely on a consensus-based gradient strategy with a switching system method to constrain DoS attacks, are capable of exactly seeking the optimal solutions under attacks over random digraphs. The global exponential convergence of the proposed algorithm can be ensured, provided that the frequency and duration of attacks satisfy certain bounded conditions;  
(2) these proposed resilient distributed optimization algorithms  avoid continuous-time communication in many existing works (see \cite{Schenato16TAC,Gharesifard14TAC,Hu16TAC,Hu19TCNS,Liu15TAC,Zhao17TAC}). The proposed dynamic event-based distributed optimization scheme is proven to be free of Zeno behavior, and  avoid fixed and periodic transmissions used in the time-based scheme; and (3) in contrast to related optimization works that require known positive Hessians of the local cost functions \cite{Hu19TCNS} and \cite{Schenato16TAC}, known gradients satisfying special linear structures \cite{Zhao17TAC}, known initial agent states \cite{Kia15AT,LiuTCNS18,Ding20TAC},
the proposed algorithms in this work remove those limitations to facilitate practical applications. Moreover, compared to  optimization works in \cite{Sundaram15Conf,Su16ACC,Sundaram19TAC,Zhao20TAC} that consider Byzantine faults on nodes and require the removal of their states to be known a prior, these DoS attacks on communication networks are time-sequence based and allowed to occur aperiodically. Another contribution of this paper is that unlike works in  \cite{Schenato16TAC,Gharesifard14TAC,Hu16TAC,Hu19TCNS,Zhao17TAC,Kia15AT,Liu15TAC,LiuTCNS18,Ding20TAC} and  \cite{Sundaram15Conf,Su16ACC,Sundaram19TAC,Zhao20TAC} over the  undirected graph or weight-balanced digraph, the studied graphs are directed and time-varying, and under cyber-attacks, the graphs can be even disconnected or totally paralyzed.

 
\vspace*{-1pt} 
\section{Preliminaries and Problem Formulation} 
\vspace*{-2pt} 
\subsection{Notation \label{Notation}}
\vspace*{-2pt} 
Denote $\mathbb{R}$, $\mathbb{R}^{n}$, and $\mathbb{R}^{n\times m}$ as the sets of the real numbers, real $n$-dimensional vectors and real $n\times m$ matrices, respectively. 
Let $ \mathbb{N}^{+} $ be the set of positive natural numbers.  
Let $\textbf{0}$ ($\textbf{1}$) be the vector with all zeros (ones) with proper dimensions. 
Let col$(x_{1},...,x_{n})$ and diag$\{a_{1},...,a_{n}\}$ be a column vector with the entries $x_{i}$ and a diagonal matrix with the entries $a_{i}$, $i=1,2,\cdots ,n$, respectively. 
$\otimes $ and $\left\Vert \cdot \right\Vert $ represent the Kronecker product and Euclidean norm, respectively. For a real matrix $M=M^{T}$, let $ M > 0 $ be positive definite. Let $\lambda _{\min }(M)$, $\lambda _{\max }(M)$ be its minimum and maximum eigenvalues, respectively. Besides, $\sigma _{\max}(M)$ represents the maximum singular value of a matrix $ M $. 
For a differentiable function $ f: \mathbb{R}^{n} \rightarrow  \mathbb{R} $, $ \triangledown f $ is the gradient of $ f $, and $ f $ is strongly convex over a convex set $ \mathbb{R}^{n} $,  
if  $ (x-y)^{T}( \triangledown f(x)- \triangledown f(y)) > \iota ||x-y||^{2}$ for $ \forall x,y \in \mathbb{R}^{n}, x \neq y$ and a scalar $ \iota>0$. $ f $ is locally Lipschitz at $ x \in \mathbb{R}^{n} $ if there exists a neighbourhood $ \mathcal{W} $ and a scalar $ l $ so that $ ||f(x)-f(y)|| \leq l ||x-y|| $ for $\forall x, y \in \mathcal{W} $; $ f $ is locally Lipschitz on $ \mathbb{R}^{n}$ if it is locally Lipschitz at $ x $ for $\forall x \in \mathbb{R}^{n} $.

\vspace*{-4pt}
\subsection{Graph Theory\label{Graph theory}} 
\vspace*{-2pt}
\textit{Static Digraph:} let $\mathcal{G}$ $=$ $\left\{ \mathcal{V},\mathcal{E}\right\} $ be a graph and $\mathcal{V}$ $\in $ $\left\{ 1,...,N\right\} $ be the set of vertices. The set of edges is denoted as $\mathcal{E}$ $\subseteq$ $ \mathcal{V\times V}$. 
$\mathcal{N}_{i}$ $=$ $%
\left\{ j\in \mathcal{V\mid }(j,i)\in \mathcal{E}\right\} $ is the neighborhood set of vertex $i$. For a directed graph $\mathcal{G}$, $(i,j)\in \mathcal{E}$ means that the
information of node $ i $ is accessible to node $ j $, but not conversely. A matrix $A$ $=$ $\left[ a_{ij}\right] $ 
is the adjacency matrix of $\mathcal{G}$, where $a_{ij}>0$ if $(j,i)\in \mathcal{E}$, else $a_{ij}=0$. A matrix $\mathcal{L}=[l_{ij}] $ is called the Laplacian matrix of $\mathcal{G}$, where $ l_{ii}=\sum^{N}_{j=1}a_{ij} $ and $ l_{ij}= -a_{ij}$, $ i \neq j$. 

\vspace*{1pt}
\textit{Markovian Random Digraph:} 
let $\mathcal{G}(t)=\left\{ \mathcal{V},\mathcal{E}_{r(t)}\right\} $ be a time-varying digraph with
 $\mathcal{E}_{r(t)}$ being a set of edges, and $r(t): [0, \infty) \\ \rightarrow \mathcal{S}=\{1,2,...,s\}$ is a piecewise constant function with $ \mathcal{S} $ being an index set of possible digraphs. The piecewise-constant function $r(t)$ is a Markovian signal. $\mathcal{A}(t)$ $=[ a_{ij}^{r(t)} ] $ is the adjacency matrix, where $%
a_{ij}^{r(t)}>0$ if $(j,i)\in \mathcal{E}_{r(t)}$, else $a_{ij}^{r(t)}=0.$ The neighboring set is denoted by $
\mathcal{N}_{i}(t)=\left\{ j \in \mathcal{V}, (j,i) \in \mathcal{E}_{r(t)}  \right\} $. Denote  $\mathcal{L}(t)=[l^{r(t)}_{ij}]$, where  $ l^{r(t)}_{ii}=\sum^{N}_{j=1}a^{r(t)}_{ij} $ and $ l^{r(t)}_{ij}= -a^{r(t)}_{ij}$, $ i \neq j$.

\vspace*{-2pt}    
\subsection{Heterogeneous Linear Multi-Agent Model}
\vspace*{-1pt}
Consider a multi-agent system consisting of $ N $ agents governed by the following heterogeneous linear dynamics:
\vspace*{-2pt}    
\begin{equation}  
\dot{x}_{i}(t)=A_{i}x_{i}(t)+B_{i}u_{i}(t), \ y_{i}(t)=C_{i}x_{i}(t), \ i \in \mathcal{V}, 
\label{Dynamics} 
\end{equation}
where $ x_{i} \in \mathbb{R}^{n_{i}} $ denotes the state of agent $i$, $ u_{i} \in \mathbb{R}^{p_{i}} $ denotes the control input of agent $i$, $ y_{i} \in \mathbb{R}^{q} $ is its output, and $ A_{i} \in \mathbb{R}^{n_{i} \times n_{i}}$, $B_{i} \in \mathbb{R}^{n_{i} \times p_{i}}$,  $C_{i} \in \mathbb{R}^{q \times n_{i}}$ are constant system matrices. Suppose that the matrix pair $ (A_{i},B_{i}) $ is controllable and 
\vspace*{-1pt}
\begin{equation}
\hspace{2em}	
\text{rank} \left[ \begin{array}{cc}
C_{i}B_{i} & 0_{q\times p_{i}} \\
-A_{i}B_{i} & B_{i} \\
\end{array} \right] = n_{i}+q, \ i \in \mathcal{V}. 
\label{0assumption} 
\end{equation}	
  
\textit{General Distributed Optimization Problem}: design a distributed scheme $ u_{i}(t) $ for (\ref{Dynamics}) using local interaction and information over a communication network so that the output of all agents can reach the optimal state $\theta^{*}$ that minimize: $ \mathrm{F}(\theta)=\sum_{i=1}^{N}f_{i}(\theta), \ \theta \in \mathbb{R}^{q}$,  
where $ f_{i}(\theta): \mathbb{R}^{q}\rightarrow \mathbb{R} $ is the private cost function known to agent $ i $ only, and $ \theta $ is the global decision variable to be optimized. From  \cite{Gharesifard14TAC}, it is equivalent to solve:  
\vspace*{-5pt}
\begin{equation} \label{OCP} 
\begin{split} 
& \min \limits_{y \in \mathbb{R}^{Nq} }  \ \tilde{f}(y)=\sum_{i=1}^{N}f_{i}(y_{i}), \  y_{i} \in \mathbb{R}^{q},  \\
\text{subject to} \ & (\ref{Dynamics}) \ \text{and} \   
y_{i}=y_{j}, \ \forall i,j \in \mathcal{V}=\{1,2,\cdots,N\}, 
\end{split}
\end{equation}
where $ y_{i} \in \mathbb{R}^{q} $ is a local estimate on the optimal solution $\theta^{*}$, and $ y=\text{col}(y_{1},\cdots,y_{N}) $ is its stack vector of all estimates.  

\vspace{3pt}
To solve the problem, two standard assumptions are introduced.

\vspace{2pt}
\begin{condition} \label{Assumption2}
There exists $ y^{*}=1_{q}\otimes \theta^{*} $ that minimizes the team cost function, i.e., $ \tilde{f}(y^{*})=\min_{\theta \in \mathbb{R}^{q}} F(\theta)$. 	
\end{condition}

\vspace{2pt}
\begin{condition} \label{Assumption3}
Each function $ f_{i}: \mathbb{R}^{q}\rightarrow \mathbb{R} $ is differentiable, strongly convex, and its gradient is locally Lipschitz on $ \mathbb{R}^{q} $. 
\end{condition}

\begin{remark}
Assumption \ref{Assumption2} guarantees the optimal solution set is nonempty. By Assumption \ref{Assumption3}, $ ||\triangledown f_{i}(x_{i})- \triangledown f_{i}(y_{i})|| \leq l_{i} ||\triangledown  x_{i}- y_{i}||, \ \forall x_{i}, y_{i} \in \mathbb{R}^{q} $, where $ \triangledown f_{i}(x_{i}) $ and $  \triangledown f_{i}(y_{i}) $ are the gradients, and $ l_{i}>0  $ is the Lipschitz constant. 
\end{remark}

\vspace*{-5pt}
\subsection{Unreliable Random Communication Network}
\vspace*{-1pt}
In large-scale cyber-physical systems, the wireless communication may not be reliable due to certain  
physical uncertainties, e.g., failures, quantization errors, and packet losses in a digital communication \cite{Zhang12TAC}. We consider an unreliable network consisting of $N$ agents whose integrated wireless communication links are time-varying and failure-prone with certain probabilities. As considered in \cite{Zhang12TAC}, a random Markov chain model can be adopted to capture this situation. This random process describes dynamic changes of topologies. Let $r(t)$ be a right-continuous homogeneous Markovian process on the probability space taking values in a finite state space $\mathcal{S}=\{1,2,...,s\}$ with an infinitesimal generator $\Upsilon =(\gamma _{pq})$, 
given by $ \mathrm{P}_{pq}(t)= \mathrm{\Pr ob}\{r(t+h)=q|r(t)=p\}  
= \gamma _{pq}h+o(h)$, if $p\neq q$, else, 
$1+\gamma _{pp}h+o(h) $,    
where $\gamma _{pq}\geq 0$ is the transition rate from the state $p$ to the state $q$, while $\gamma _{pp}= -\sum^{s}_{q=1,p\neq q}  \gamma _{pq},$ and  $o(h)$ satisfies: $\lim_{h\rightarrow 0}o(t)/h=0$.

\subsection{Malicious DoS Attack Model}
\vspace*{-1pt}
As studied in \cite{Hu19TCST}, the DoS attack refers to the interruption of the communication channels. Without loss of generality,
suppose that there exists an attack sequence $\{a_{m}\}_{m\in\mathbb{N}}$
when a DoS attack is lunched at $a_{m}$, and let the duration of this attack be $\tau_{m}\geq 0$. 
Then, the $m$-th DoS attack strategy can be generated with  $\mathcal{A}_{m}$ $=$ $ {a_{m}} \cup [a_{m},a_{m}+\tau_{m})$ with $a_{m+1}>a_{m}+\tau_{m}$ for all $m\in\mathbb{N}.$ Thus, for given $t\geq t_{0} $, the sets of time instants where communication is denied (unsuccessful) are described by 
\vspace*{-2pt}
\begin{equation}
\Xi _{a}(t_{0},t)=\cup \ \mathcal{A}_{m} \cap \ \lbrack t_{0},t], \ m\in\mathbb{N},
\label{DoSAttacks}
\end{equation}
which implies that on the interval $ [t_{0} , t] $, the sets of time instants where communication is allowed are: $\Xi _{s}(t_{0},t)=[t_{0} ,t] \setminus \Xi _{a}(t_{0}, t)$. In words, $ |\Xi _{a}(t_{0},t)| $ and $ |\Xi _{s}(t_{0},t)| $ represent the total lengths of the attacker being active and sleeping over $ [t_{0},t] $, respectively.

\vspace*{2pt}
\begin{definition} [Attack Frequency] \label{Attack Frequency}
For any $T_{2}>T_{1}\geq t_{0}$, let $N_{a}(T_{1},T_{2})$ denote the total number of DoS attacks over $[T_{1},T_{2})$. Then,   $F_{a}(T_{1},T_{2})=\frac{N_{a}(T_{1},T_{2})}{T_{2}-T_{1}}$ denotes the attack frequency over $[T_{1},T_{2})$ for $\forall T_{2}>T_{1}\geq t_{0}$, where there exists scalars $ N_{0}, T_{f} >0 $ such that $ N_{a}(T_{1},T_{2}) \leq N_{0}+ (T_{2}-T_{1})/T_{f}. $
\end{definition}

\vspace*{1pt}
\begin{definition} [Attack Duration] \label{Attack Duration}
For any $T_{2}>T_{1}\geq t_{0},$ let $|\Xi_{a}(T_{1},T_{2})|$ be the total time interval under attacks over $[T_{1},T_{2}).$ The attack duration over $[T_{1},T_{2})$ is defined as: there exist scalars $T_{0}\geq 0$, $T_{a}>1$ so that $|\Xi_{a}(T_{1},T_{2})| \leq T_{0}+(T_{2}-T_{1})/T_{a}.$
\end{definition}

\begin{remark}
As mentioned in the preliminary works  \cite{Hu15IFAC,Hu15Cyber,Hu15IJNRC,Hu19TCST} for multi-agent systems under malicious cyber-attacks, Definitions \ref{Attack Frequency} and \ref{Attack Duration} are firstly introduced in \cite{Hu15IFAC} to specify DoS attack signals in terms of the frequency and time-ratio constraints. 
\end{remark}

\vspace*{-3pt} 
\subsection{Main Objective}
\vspace*{-2pt} 
The main objective of the paper is to solve a  resilient distributed optimization issue of heterogeneous linear multi-agent systems.

\vspace*{2pt}
\begin{problem} \label{Problem2}
(\textbf{Resilient Distributed Convex Optimization against DoS Attacks over Random Digraphs}) \\
\vspace*{1pt}
Develop a resilient distributed algorithm $ u_{i} $ so that the output of each agent cooperatively seeks $ \theta^{*} $ under DoS attacks over random digraphs. The problem in (\ref{OCP}) is reformulated as  
\vspace*{-4pt}          
\begin{equation} \label{P2} 
\begin{split} 
& \min \limits_{y \in \mathbb{R}^{Nq} }  \ \tilde{f}(y)=\sum_{i=1}^{N}f_{i}(y_{i}), \  y_{i} \in \mathbb{R}^{q},  \\
\text{subject to} \ &  \dot{x}_{i}(t)=A_{i}x_{i}(t)+B_{i}u_{i}(t),   y_{i}(t)=C_{i}x_{i}(t), \\
\text{and} \ &  (\mathcal{L}(t) \otimes I_{q}) y=\textbf{0}, 
\ t \in \Xi _{s}(t_{0},t) \cup \Xi _{a}(t_{0},t).  
\end{split}
\end{equation} 
\end{problem} 
 
\begin{remark} 
Solving resilient exponential distributed optimization in Problem \ref{Problem2} is much more challenging in fourfold: (1) \textit{Agent dynamics:} we investigate general linear multi-agent systems with nonidentical dynamics and dimensions; (2) \textit{Communication:} the unreliable communication network leads to time-varying directed graphs and moreover, the interruptions of communication channels caused by DoS attacks make existing distributed optimization algorithms degraded or totally failed; (3) \textit{Assumptions:} the local gradient is nonlinear without needing to satisfy special structures in many existing works and only the local Lipschitz property is required; and 4) \textit{Design requirements:} propose the initialization-free resilient distributed optimization algorithms with the discrete-time communication nature to provide global exponential convergence and resilience features against DoS attacks over random digraphs. The existing designs in  \cite{Schenato16TAC,Gharesifard14TAC,Hu16TAC,Hu19TCNS,Zhao17TAC,Kia15AT,Liu15TAC,LiuTCNS18,Ding20TAC} cannot be directly applied.
\end{remark}

\section{Resilient Exponential Distributed Optimization Against Dos ATTACKS over Random Networks}
Due to unreliable networks described in the  subsection II-D, the underlying topologies are time-varying and random. As stated in \cite{Nedic11TAC}, define that a sequence of Laplacian matrices $ \{\mathcal{L}(t)\} $ admits a common stationary distribution $ \pi>0 $ if $ \mathcal{L}(t)\pi =\textbf{0} $. 
Let $ \mathcal{L}_{s}(t) $ be a mirror of $\mathcal{L}_{un}(t)$, i.e.,  $ \mathcal{L}_{s}(t)= (\mathcal{L}_{un} + \mathcal{L}_{un}^{T})/2 $, where $ \mathcal{L}_{un}=\sum_{p=1}^{N} L_{p}(t) $ is  Laplacian matrix of a union of digraphs. Define the minimum cut of $ \{\mathcal{L}(t)\} $ as $ l_{c}(t)=\min_{\mathcal{S} \subset \mathcal{V}, \mathcal{S} \neq \emptyset } \sum_{i \in \mathcal{S}, j\in \bar{\mathcal{S}}} \mathcal{L}_{s}(t) $, where $ \bar{\mathcal{S}} $ is the complement of $ \mathcal{S} $. Then, we say that this sequence of $ \{\mathcal{L}(t)\} $ has a minimum cut $ c $, if $ l_{c}(t) \geq c >0 $.

\vspace*{2pt}
\begin{condition} \label{Assumption4}
Assume that the sequence of Laplacian matrices $ \{\mathcal{L}(t)\} $ has a stationary distribution $ \pi $ with a minimum cut.  
\end{condition} 	
 	 
\vspace*{2pt} 	  
\begin{lemma} \label{Lemma3}
By Assumption \ref{Assumption4}, let $ \mathcal{L}(t) $ be the Laplacian matrix with a stationary distribution $ \pi>0 $. Then, there exists a weighted matrix $ Q(t)=\mathcal{L}(t)\Pi+\Pi\mathcal{L}^{T}(t) $, $ \Pi=\text{diag}\{\pi \} $, so that for $ \pi_{min} = \min_{p \in \mathcal{V}} \{ \pi_{p} \}$ and a vector $ \xi \in \mathbb{R}^{N} $ satisfying $ \pi^{T} \xi =0 $, we have 
\vspace{-4pt}
\begin{equation}
\xi^{T}Q(t)\xi = \sum_{i=1}^{N} \sum_{j\in \mathcal{N}(t)}Q_{ij}(t)(\xi_{i}-\xi_{j})^{2} \geq \frac{ \pi_{min} c }{N^{2}} ||\xi||^{2}.  
\end{equation}
\end{lemma} 	

\vspace*{-2pt} 
\begin{proof} 
without loss of generality, we suppose that $ || \xi || \neq 0 $, where it satisfies $\xi_{1} \leq \xi_{2} \leq \cdots \leq \xi_{N} $. Since $ || \xi || \neq 0 $ and $ \pi^{T} \xi =0 $, we have $ \xi_{1} <0 $ and $ \xi_{N}>0 $. Define $ i_{*} \in \text{argmax}_{i \geq 2} | \xi_{i}-\xi_{i-1} |$ and $ \varepsilon = | \xi_{i_{*}}-\xi_{i_{*}-1} |$. Then, we obtain that 
\vspace{-3pt}
\begin{align}
\xi^{T}Q(t)\xi &= \sum_{i=1}^{N} \sum_{i < j  }Q_{ij}(t)(\xi_{i}-\xi_{j})^{2} 
\geq  \sum_{i \leq i_{*}}  \sum_{i_{*} < j }  Q_{ij}(t) \varepsilon^{2} \notag \\
& \geq  \pi_{min} l_{c}(t) \varepsilon^{2} 
\geq \frac{  \pi_{min} c }{N^{2}} ||\xi||^{2}, 
\label{Quadric}  
\end{align}
where the fact $ |\xi_{i}| \leq N \varepsilon $ is used to get the last inequality.
\end{proof}
	  
\vspace{-5pt}
\subsection{Distributed Time-Based Strategy Against DoS Attacks}
\vspace{-1pt}  
Based on the DoS attack model in the subsection II-E, 
without loss of generality, suppose that there exists an infinite sequence of uniformly bounded non-overlapping time intervals $[t_{k}, t_{k+1}), k=0,1,2\cdots$. Each agent $ i $ thus updates its controller $ u_{i}(t) $ over time intervals $[t_{k}, t_{k+1}) $. Then, in the presence of the DoS attacks,
we denote the attack-induced new time sequence as $ 0=t_{0} =  a_{0} <t_{1}<a_{1}<a_{1}+\tau_{1}\leq t_{2}<\cdots<t_{k}<a_{k}<a_{k}+\tau_{k}\leq t_{k+1}<\cdots $.    

\vspace{1pt}
Next, we propose the following distributed algorithm to analyze the effect of DoS attacks over random digraphs.
\vspace{-3pt}
\begin{subequations}\label{Controller2}	
	\begin{align}
	u_{i}&=-K_{i}x_{i}-(U_{i}-K_{i}X_{i})\varrho_{i}+W_{i}\vartheta_{i},  \label{TC1} \\
	\dot{\varrho}_{i}&=\vartheta_{i}= -\triangledown f_{i}(y_{i}) -\beta e_{\varrho zi} - \alpha \beta e_{y i} , \label{TC2} \\
	\dot{z}_{i} & = \alpha \beta e_{y i},  t \in   [t_{k},t_{k+1}), \ i \in \mathcal{V},  \label{TC3} 
	\end{align}
\end{subequations}
where $ \alpha,\beta \in \mathbb{R} $ are positive constant gains, $\vartheta_{i}$ is the intermediate variable, and $ e_{\varrho zi} $, $ e_{yi} $ are  consensus errors under DoS attacks 
\vspace{-2pt}
\begin{equation}
e_{\varrho z i} = \left\{
\begin{array}{c} 
\sum_{j\in \mathcal{N}_{i}(t) }a_{ij}(t) (\varrho_{i}-\varrho_{j} + z_{i}-z_{j}),  \text{ if} \ t \in [t_{k},a_{k}), \\
\hspace{-5em}
\textbf{0},  \text{ if} \ t \in   [a_{k},t_{k+1}) , \ k=0,1,2\cdots,  
\end{array} 
\right.   \label{TC4} 
\end{equation}
\vspace{-12pt}
\begin{equation}
e_{y i} = \left\{
\begin{array}{c} \sum_{j\in \mathcal{N}_{i}(t)} a_{ij}(t) (y_{i}-y_{j}),  \text{ if} \ t \in [t_{k},a_{k}), \\
\hspace{-1em}
\textbf{0},  \text{ if} \ t \in   [a_{k},t_{k+1}), \  k=0,1,2\cdots, 
\end{array}%
\right. \label{TC5}
\end{equation}
where  $ \varrho_{i}, z_{i} \in \mathbb{R}^{q}$ denote the auxiliary states. In (\ref{Controller2}),  $ K_{i} \in \mathbb{R}^{p_{i}\times n_{i}} $ 
and $ (U_{i}, W_{i},  X_{i})$ is the solution to the following equations: 
\vspace{-3pt}
\begin{equation}
B_{i}U_{i}=A_{i}X_{i}, \ B_{i}W_{i}=X_{i}, \ C_{i}X_{i} =I_{q}, \ i \in \mathcal{V}. 
\label{RegulationEquation}
\end{equation}

\begin{remark}
The proposed algorithm in (\ref{Controller2})-(\ref{RegulationEquation}) has properties: 1) it is distributed as each agent communicates with its neighbors only; 2) each agent transmits exchange information over random digraphs, while in the presence of DoS attacks, the information is interrupted; and 3) $ \sum_{j\in \mathcal{N}_{i}(t) }a_{ij}(t) ( z_{i}-z_{j})  $ in $ e_{\varrho z i} $ is  relative internal state information, which is used to avoid zero-sum initial requirements.
Besides, the solution of (\ref{RegulationEquation}) is ensured by (\ref{0assumption}).      
\end{remark}

\vspace{2pt} 
Define $ \tilde{\mathcal{L}}(t)=\mathcal{L}(t)\otimes I_{q} $. Substituting (\ref{Controller2})-(\ref{TC5}) into (\ref{Dynamics}) gives the following closed-loop system under attacks and random digraphs.  
\vspace{-4pt}
\begin{equation} \label{Controller2Compact}
\dot{\chi} = 
\left\{ 
\begin{array}{l}
\left[ \begin{array}{ccc}
(A-BK)x+BW \vartheta-B(U-KX)\varrho  \\
-\triangledown \tilde{f}(y) - \alpha \beta \tilde{\mathcal{L}}(t) y - \beta \tilde{\mathcal{L}}(t) (\varrho + z) \\
\alpha \beta  \tilde{\mathcal{L}}(t) y
\end{array} \right],   \\   \text{ if} \ t \in  [t_{k},a_{k}),
k=0,1,2\cdots   \\
\left[ \begin{array}{ccc}
(A-BK)x+BW \vartheta-B(U-KX)\varrho  \\
-\triangledown \tilde{f}(y)  \\
\textbf{0}
\end{array} \right],   \\  \text{ if} \ t \in  [a_{k},t_{k+1}) ,
k=0,1,2\cdots
\end{array}
\right. 	
\end{equation} 
where $\chi=\text{col} (x,\varrho,z) $ and $ x, \varrho, z, \vartheta, \triangledown \tilde{f}(y) $ are stacked vectors of $ x_{i} $, $\varrho_{i}$, $z_{i}$,  $\vartheta_{i} $, $ \triangledown f_{i}(y_{i}) $, and $ A, B, C, K, U, W, X $ are stacked block diagonal matrices of $ A_{i} $, $ B_{i} $, $ C_{i} $, $ K_{i} $, $ U_{i} $, $ W_{i} $, $ X_{i} $, respectively. 

\vspace{1pt} 
Next, we present the exponential resilient distributed optimization result against DoS attacks and random digraphs.
 
\vspace{2pt} 
\begin{theorem} \label{Theorem2}
Given Assumptions \ref{Assumption2}-\ref{Assumption4}, Problem 1 is solvable for any $ x_{i}(0)$, $\varrho_{i}(0) $, and $z_{i}(0) $ under the proposed resilient distributed optimization algorithm in (\ref{Controller2})-(\ref{TC5}), provided that $ K_{i} $ is chosen so that $A_{i}-B_{i}K_{i}$ is Hurwitz, $ U_{i}, W_{i},  X_{i}$ are solutions to (\ref{RegulationEquation}), and for scalars $\lambda_{a}, \lambda_{b} , u >0 $ to be determined later, the following two arrack-related conditions are satisfied: 

\vspace*{1pt}
(1). There exists constants $\eta^{\ast }\in (0,\lambda_{a})$ and $ \mu \geq 1 $ so that $T_{f}$ in the \textit{attack frequency} Definition \ref{Attack Frequency} satisfies the condition:
\vspace*{-3pt}
\begin{equation}
T_{f} > T^{*}_{f}=  \ln(\mu)/\eta^{\ast},
\label{Condition1}
\end{equation}

\vspace{-3pt}
(2). There exist constants $ \lambda_{a}$, $\lambda_{b}>0 $ such that $T_{a}$ in the \textit{attack duration} Definition \ref{Attack Duration} satisfies the condition: 
\vspace*{-3pt}
\begin{equation}
\hspace{2em}
T_{a}>T^{*} _{a}=(\lambda_{a}+\lambda_{b})/(\lambda_{a}-\eta^{\ast }). 
\label{Condition2}
\end{equation}%

\vspace{-3pt}
Moreover, the estimated states converge exponentially, i.e.,  
\vspace*{-2pt}
\begin{equation}
\| \tilde{\chi} (t) \|^{2} \leq \varsigma e^{-\eta (t-t_{0})} ||\tilde{\chi}(t_{0}) ||^{2},  \ \forall t_{0} \geq 0,
\label{ExponentialConvergence}
\end{equation}%
where $ \tilde{\chi}= \text{col} (\tilde{x}-X\tilde{\varrho},\tilde{\varrho}  + \tilde{z} , \tilde{y}) $ with $ \tilde{x}=x-\bar{x} $, $ \tilde{\varrho}=\varrho-\bar{\varrho} $, and $ \tilde{z}=z-\bar{z} $ representing the state transformations, $ \varsigma $ is a positive scalar to be determined later, and 
$ \eta=\lambda_{a}-(\lambda_{a}+\lambda_{b})/T_{a} -\eta^{*}>0$. 
\end{theorem}

\vspace{3pt} 
\begin{proof} 
the idea is to show the convergence of $(x,\varrho,z)$ to the equilibrium point, which will include four steps below: \\
\vspace{-1pt}
i) consider the case that the communication is not subject to DoS attacks, i.e., $ t \in [t_{k},a_{k}) $. We first show that the output $ y_{i} $ at the equilibrium point is an optimal solution of (\ref{P2}).   

\vspace{1pt}
In the absence of DoS attacks, (\ref{Controller2Compact}) can be rewritten as
\vspace{-4pt}
\begin{subequations}\label{Controller1Compact}	
\begin{align}
\dot{x}&=(A-BK)x+BW \vartheta-B(U-KX)\varrho,  \label{CC1} \\
\dot{\varrho}&=\vartheta=-\triangledown \tilde{f}(y) - \alpha \beta \tilde{\mathcal{L}}(t) y - \beta \tilde{\mathcal{L}}(t) (\varrho  + z) , \label{CC2} \\
\dot{z} & = \alpha \beta  \tilde{\mathcal{L}}(t) y. \label{CC3}
\end{align}
\end{subequations}
 
\vspace{-3pt}
By (\ref{RegulationEquation}), we obtain the equilibrium point $ (\bar{x},\bar{\varrho},\bar{z}) $ from 
\vspace{-3pt}
\begin{subequations}\label{Controller1CompactEquilibrium}
\begin{align}
\dot{\bar{x}}&=\textbf{0} \ \Longrightarrow \ (A-BK) (\bar{x}-X\bar{\varrho}) + X \dot{\bar{\varrho}}=\textbf{0},  \label{CCE1} \\
\dot{\bar{\varrho}}&=\textbf{0} \ \Longrightarrow \    -\triangledown \tilde{f}(y) - \alpha \beta \tilde{\mathcal{L}}(t) \bar{y} - \beta \tilde{\mathcal{L}}(t) (\bar{\varrho}  + \bar{z}) =\textbf{0} , \label{CCE2} \\
\dot{\bar{z}}&=\textbf{0} \ \Longrightarrow \   \alpha \beta  \tilde{\mathcal{L}}(t) \bar{y} = \alpha \beta  \tilde{\mathcal{L}}(t) C \bar{x} =\textbf{0}.  \label{CCE3}
\end{align}
\end{subequations}

In the sequel, we show that the equilibrium point is the solution. Deducing from (\ref{Controller1CompactEquilibrium}), the equilibrium point satisfies
\vspace{-3pt}
\begin{equation}
(A-BK) (\bar{x}-X\bar{\varrho}) = \textbf{0}, \  \sum_{i=1}^{N} \triangledown \tilde{f}_{i}(\bar{y}_{i}) = \textbf{0}, \ \bar{y}_{i}=C_{i}\bar{x}_{i}=y^{*}, \label{Controller1CompactEquilibriumCondition}
\end{equation}
where $ y^{*} \in \mathbb{R}^{q} $, and since $ \tilde{f} $ is strongly convex, $ \sum_{i=1}^{N} \triangledown \tilde{f}_{i}(y^{*}) = \textbf{0} $ implies that $ y^{*}  $ is the optimal solution. Since $ K $ is selected such that $ A-BK $ is Hurwitz, it follows from (\ref{RegulationEquation}) and  (\ref{Controller1CompactEquilibriumCondition}) that 
\vspace{-3pt} 
\begin{equation}
\bar{x}=X\bar{\varrho} \ \ \text{and} \ \ \bar{y} = \bar{\varrho}=\textbf{1} \otimes y^{*}. 
\end{equation}

\vspace{-2pt} 
Thus, $ (\bar{x},\bar{\varrho},\bar{z})=(X\bar{\varrho},\bar{\varrho},\bar{z}) $, and $ y_{i} $ at the equilibrium point is the optimal solution of (\ref{P2}) in the absence of attacks.  

\vspace{3pt}
Let $ F= \triangledown \tilde{f}(\tilde{y}+\bar{y})-\triangledown \tilde{f}(\bar{y}) $ and $ \tilde{y}=C \tilde{x} $. Since $ \tilde{x}=x-\bar{x} $, $ \tilde{\varrho}=\varrho-\bar{\varrho} $, $ \tilde{z}=z-\bar{z} $, we further have the following system  
\vspace{-2pt} 
\begin{subequations}\label{Controller1ClosedLoop}
\begin{align}
\dot{\tilde{x}}&= (A-BK) (\tilde{x}-X\tilde{\varrho}) + X \dot{\tilde{\varrho}},  \label{CCL1} \\
\dot{\tilde{\varrho}}&=  -F - \alpha \beta \tilde{\mathcal{L}}(t) \tilde{y} - \beta \tilde{\mathcal{L}}(t) (\tilde{\varrho}  + \tilde{z}), \label{CCL2} \\
\dot{\tilde{z}}&=   \alpha \beta  \tilde{\mathcal{L}}(t) \tilde{y}.  \label{CCL3} 
\end{align}
\end{subequations}      

Choose the following nonnegative Lyapunov function candidate $ V_{a}=V_{1a}+V_{2a}+V_{3a} $, where 
\vspace{-1pt} 
\begin{subequations}\label{Lyapunovfunction}
\begin{align}
V_{1a}&= (\tilde{x}-X\tilde{\varrho})^{T} P_{a}(\tilde{x}-X\tilde{\varrho}) ,  \label{L1} \\
V_{2a}&=  \frac{1}{2} (\tilde{\varrho}  + \tilde{z})^{T} (\Pi \otimes I_{q}) (\tilde{\varrho}+\tilde{z}), \label{L2} \\
V_{3a}&= \frac{1}{2} \tilde{y}^{T} (\Pi \otimes I_{q})\tilde{y}, \label{L3} 
\end{align}
\end{subequations} 
where $ P_{a}>0 $ and $ \Pi =\text{diag}\{\pi\} $ is defined in Lemma \ref{Lemma3}. 

\vspace{2pt}
Then, the time derivatives of (21a) along (\ref{Controller1ClosedLoop}) is described by
\vspace{-2pt}  
\begin{align}
\dot{V}_{1a}&=(\tilde{x}-X\tilde{\varrho})^{T} (P_{a}\tilde{A} +\tilde{A}^{T}P_{a} ) (\tilde{x}-X\tilde{\varrho}) \label{V1} \\
&=-(\tilde{x}-X\tilde{\varrho})^{T} O (\tilde{x}-X\tilde{\varrho}) \leq  - \epsilon_{0} (\tilde{x}-X\tilde{\varrho})^{T}  (\tilde{x}-X\tilde{\varrho}), \notag
\end{align}
where $ \tilde{A}=A-BK $ is Hurwitz so that there is a positive define matrix $ O$ satisfying $ P_{a}\tilde{A} +\tilde{A}^{T}P_{a}=- O$ and $ \epsilon_{0}=\lambda_{min}(O)>0 $.
Then, it follows from (\ref{Lyapunovfunction}) and (\ref{CCL2})-(\ref{CCL3})  that 
\vspace{-2pt}  
\begin{align}
\dot{V}_{2a}& = (\tilde{\varrho}  + \tilde{z})^{T} (\Pi \otimes I_{q}) ( \dot{\tilde{\varrho}} + \alpha \beta  \tilde{\mathcal{L}}(t) \tilde{y} ),  \notag \\
& = (\tilde{\varrho}  + \tilde{z})^{T} (\Pi \otimes I_{q}) ( -F  - \beta \tilde{\mathcal{L}}(t) (\tilde{\varrho}  + \tilde{z}) ) \notag  \\
&= -(\tilde{\varrho}  + \tilde{z})^{T} (\Pi \otimes I_{q})F-\beta (\tilde{\varrho}  + \tilde{z})^{T} Q(t) (\tilde{\varrho}  + \tilde{z}).  \label{V2} 
\end{align} 

The time derivative of (\ref{L3}) is expressed as 
\vspace{-3pt}   
\begin{align}
\dot{V}_{3a}& = \tilde{y}^{T} (\Pi \otimes I_{q}) C \dot{\tilde{x}} = \tilde{y}^{T} (\Pi \otimes I_{q}) C (\tilde{A} (\tilde{x}-X\tilde{\varrho}) + X \dot{\tilde{\varrho}}) \notag \\
&= \tilde{y}^{T} (\Pi \otimes I_{q}) C\tilde{A} (\tilde{x}-X\tilde{\varrho}) - \tilde{y}^{T} (\Pi \otimes I_{q})  F  \notag  \\
& \ \ \ + \tilde{y}^{T} (\Pi \otimes I_{q})  C X (  - \alpha \beta \tilde{\mathcal{L}}(t) \tilde{y} - \beta \tilde{\mathcal{L}}(t) (\tilde{\varrho}  + \tilde{z})) \notag \\
& \leq \epsilon_{1} ||\Pi C ||^{2}  \tilde{y}^{T}\tilde{y}+ \frac{||\tilde{A}||^{2}}{4 \epsilon_{1}}  (\tilde{x}-X\tilde{\varrho})^{T}(\tilde{x}-X\tilde{\varrho})  \notag \\ 
& \ \ \ +  \tilde{y}^{T} (\Pi \otimes I_{q}) (-F  - \alpha \beta \tilde{\mathcal{L}}(t) \tilde{y} - \beta \tilde{\mathcal{L}}(t) (\tilde{\varrho}  + \tilde{z})), 
\label{V3}  
\end{align} 
where $ C_{i}X_{i}=I_{q} $ in (\ref{RegulationEquation}) and the Young inequality are used to derive the last inequality, and $ \epsilon_{1} $ is a positive constant. 

\vspace{1pt} 
Let $ \iota_{min} =\min_{i\in \mathcal{V}} \{\iota_{i}\} $, $ l_{max} =\max_{i\in \mathcal{V}} \{l_{i}\} $ and $ \epsilon_{2} >0$ is a scalar. Due to the fact that $ F= \triangledown \tilde{f}(\tilde{y}+\bar{y})-\triangledown \tilde{f}(\bar{y}) $ and according to Assumption \ref{Assumption3}, we can further obtain that 
\vspace{-1pt}   
\begin{subequations}\label{V4}
\begin{align}
& \tilde{y}^{T} (\Pi \otimes I_{q}) F  \geq \iota_{min} \pi_{min} || \tilde{y} ||^{2}, \ ||F|| \leq l_{max} ||\tilde{y}||, \label{V4a}  \\ 
& (\tilde{\varrho}  + \tilde{z})^{T} (\Pi \otimes I_{q})F \leq  \frac{||(\Pi \otimes I_{q})F ||^{2}}{4 \epsilon_{2}} + \epsilon_{2} ||\tilde{\varrho}  + \tilde{z}||^{2}. \label{V4b} 
\end{align} 
\end{subequations} 

Thus, combining (\ref{V1})-(\ref{V4}) gives rise to 
\vspace{-3pt}   
\begin{align}
\dot{V}_{a}&  \leq    [  \frac{||\tilde{A}||^{2}}{4 \epsilon_{1}} - \epsilon_{0} ]   (\tilde{x}-X\tilde{\varrho})^{T}(\tilde{x}-X\tilde{\varrho}) + \epsilon_{1}||(\Pi \otimes I_{q} ) C ||^{2}  \tilde{y}^{T}\tilde{y}  \notag \\
&  - \iota_{min} \pi_{min} \tilde{y}^{T}  \tilde{y} -  \alpha \beta  \tilde{y}^{T} Q(t) \tilde{y} - \beta  \tilde{y}^{T} (\Pi \otimes I_{q}) \tilde{\mathcal{L}}(t) (\tilde{\varrho}  + \tilde{z})) \notag \\
&  + \frac{l^{2}_{max}\pi^{2}_{max}}{4 \epsilon_{2}} ||\tilde{y}||^{2}+ \epsilon_{2} ||\tilde{\varrho}  + \tilde{z}||^{2} -\beta (\tilde{\varrho}  + \tilde{z})^{T} Q(t) (\tilde{\varrho}  + \tilde{z}) \notag \\
& \leq -  \left[ \begin{array}{ccc}
\tilde{x}-X\tilde{\varrho}  \\
\tilde{\varrho}  + \tilde{z}   \\ 
\tilde{y}    
\end{array} \right] ^{T} \varOmega_{a}  \left[ \begin{array}{ccc}
\tilde{x}-X\tilde{\varrho}  \\
\tilde{\varrho}  + \tilde{z}   \\ 
\tilde{y}    
\end{array} \right] \leq - \lambda_{a} V_{a}, 
\label{V5}  
\end{align}
where $ \lambda_{a} = \lambda_{min}(\varOmega_{a} )/\lambda_{max}(\varGamma_{a})$, $ \varGamma_{a}=\frac{1}{2} \text{diag} \{2P_{a}, \Pi\otimes I_{q}, \Pi \otimes \\    I_{q} \} $, $ \varOmega_{a} = \left[ \begin{array}{ccc}
\varOmega^{1}_{a} & \textbf{0}  &  \textbf{0}  \\
* &  \varOmega^{2}_{a}  &   \varOmega^{3}_{a}   \\ 
* &  *  &   \varOmega^{4}_{a}    
\end{array} \right] >0 $,  $ \varOmega^{1}_{a}= ( \epsilon_{0}-\frac{||\tilde{A}||^{2}}{4 \epsilon_{1}})\textbf{ I} $, $ \varOmega^{2}_{a}= (\beta \frac{ \pi_{min} c }{N^{2}} -\epsilon_{2} )\textbf{I} $, $ \varOmega^{3}_{a}= \frac{\beta}{2}(\Pi\otimes I_{q}) \tilde{\mathcal{L}}  $,  $ \varOmega^{4}_{a}= (\alpha \beta \frac{ \pi_{min} c }{N^{2}} +l_{min} \pi_{min} - \epsilon_{1}||(\Pi \otimes I_{q} ) C ||^{2} - \frac{l^{2}_{max}\pi^{2}_{max}}{4 \epsilon_{2}} ) \textbf{I}$,  where $ \pi_{max} =\max_{i\in \mathcal{V}} \{\pi_{i}\} $ and $ \textbf{I} $ is the identify matrix with proper dimensions.

Hence, denote $ \tilde{\chi}= \text{col} (\tilde{x}-X\tilde{\varrho},\tilde{\varrho}  + \tilde{z} , \tilde{y}) $ and there exists scalars $ \epsilon_{i}>0$, $i=0,1,2$ and matrices $ \varGamma_{a}=\text{diag}\{P_{a}, \Pi \otimes I_{q}, \Pi \otimes I_{q} \} $, $ P_{a}>0 $, $ \Pi=\text{diag}\{ \pi \} $, such that we have  
\vspace{-3pt}
\begin{equation}
V_{a}=\frac{1}{2} \tilde{\chi}^{T}\varGamma_{a} \tilde{\chi} \  \Rightarrow  \ \dot{V}_{a} \leq - \lambda_{a} V_{a}, \ t \in [t_{k},a_{k}). \label{VV1}
\end{equation}   

ii) consider the case that in the presence of DoS attacks, i.e., $ t \in [a_{k},t_{k+1}) $, we choose the Lyapunov function candidate as
\vspace{-3pt}
\begin{equation}
V_{b}=\frac{1}{2} \tilde{\chi}^{T}\varGamma_{b}\tilde{\chi},  \  \varGamma_{b} = \left[ \begin{array}{ccc}
2P_{b} & \textbf{0} & \textbf{0} \\
* & S \otimes I_{q} & \textbf{0} \\
* & *  & S \otimes I_{q}
\end{array} \right],  \label{VV2} 
\end{equation} 
where $ P_{b}>0 $, and $ S>0 $ is a diagonal  matrix.

Then, the time derivative of $ V_{b} $ along (\ref{Controller2Compact}) is given by  
\vspace{-2pt}
\begin{align}
\dot{V}_{b}&=(\tilde{x}-X\tilde{\varrho})^{T} ( P_{b}\tilde{A} +\tilde{A}^{T}P_{b} ) (\tilde{x}-X\tilde{\varrho}) -(\tilde{\varrho}  + \tilde{z})^{T} (S \otimes I_{q})F \notag \\
& \ \ \ + \tilde{y}^{T} (S \otimes I_{q}) C\tilde{A} (\tilde{x}-X\tilde{\varrho}) - \tilde{y}^{T} (S \otimes I_{q})  F \notag \\
& \leq [ \frac{||\tilde{A}||^{2}}{4 \tilde{\epsilon}_{1}} - \tilde{\epsilon}_{0}]   (\tilde{x}-X\tilde{\varrho})^{T}(\tilde{x}-X\tilde{\varrho}) + \tilde{\epsilon}_{1}||(S \otimes I_{q} ) C ||^{2}  \tilde{y}^{T}\tilde{y}  \notag \\
& \ \ \ - \iota_{min} \lambda_{min}(S) \tilde{y}^{T}  \tilde{y} + \frac{l^{2}_{max}}{4 \tilde{\epsilon}_{2}} \tilde{y}^{T}\tilde{y}+ \tilde{\epsilon}_{2} (\tilde{\varrho}  + \tilde{z})^{T}(\tilde{\varrho}  + \tilde{z})   \notag \\
& \leq \tilde{\chi}^{T} \varOmega_{b} \tilde{\chi}  \leq   \lambda_{b} V_{b} ,  
\label{VV3}  
\end{align}
where $ \lambda_{b}=\sigma_{max} ( \varOmega_{b} ) >0 $, and $ \varOmega_{b} $ can be described by  
\vspace{-3pt}
\begin{equation}
\varOmega_{b}= \left[ \begin{array}{ccc}
(\frac{||\tilde{A}||^{2}}{4 \tilde{\epsilon}_{1} }-\tilde{\epsilon}_{0})\textbf{ I}  & \textbf{0} & \textbf{0} \\
* &  \tilde{\epsilon}_{2} \textbf{I} & \textbf{0} \\
* & *  & \tilde{\epsilon}_{3} \textbf{I} 
\end{array} \right],  \label{VV4} 
\end{equation}
where $ \tilde{\epsilon}_{0}=\lambda_{min}(\tilde{O}) $ with  $P_{b}\tilde{A} +\tilde{A}^{T}P_{b} =-\tilde{O}$, $ \tilde{\epsilon}_{3}=\tilde{\epsilon}_{1}||(S \otimes I_{q} ) \\ C ||^{2} + \frac{l^{2}_{max}}{4 \tilde{\epsilon}_{2}} -\iota_{min} \lambda_{min}(S) $, and $\tilde{\epsilon}_{1}, \tilde{\epsilon}_{2}, \tilde{\epsilon}_{3}>0 $.

\vspace{3pt}
iii) we analyze the exponential convergence of the closed-loop system from a hybrid perspective in  \cite{Hu15IFAC,Hu15Cyber,Hu15IJNRC,Hu19TCST}. 

Denote $ \delta(t) \in \{a,b\} $ as a switching signal. Then, we can select a piecewise Lyapunov function candidate: $ V(t)=
V_{a}(\tilde{\chi})$, if  $ t \in  [t_{k},a_{k})$, and  $V(t)=
V_{b}(\tilde{\chi})$, if $ t \in   [a_{k},t_{k+1}), $
 where $ V_{a}(\tilde{\chi}) $ and $ V_{b}(\tilde{\chi}) $ are defined in (\ref{VV1}) and (\ref{VV2}), respectively.
 
We suppose that $ V_{a} $ is activated in $ [t_{k},a_{k}) $, while $ V_{b} $ is activated in $ [a_{k},t_{k+1}) $. Then, it follows from  (\ref{VV1}) and (\ref{VV3}) that we get  
\vspace{-3pt}
\begin{equation}
V(t)=\left\{ 
\begin{array}{c}
\hspace{-0.7em}  
e^{-\lambda_{a}(t-t_{k}) } V_{a}(t_{k}), \  \text{if} \ t \in [t_{k},a_{k}), \\ 
e^{\lambda_{b}(t-a_{k}) } V_{b}(a_{k}), \  \text{ if}  \  t \in [a_{k},t_{k+1}).
\end{array}%
\right. 
\label{VV6}
\end{equation}

Note that a closed-loop system is switched at $ t= t^{+}_{k} $ or $ t= a^{+}_{k} $. Let $ \mu=\max\{\lambda_{max}(\varGamma_{a})/\lambda_{min}(\varGamma_{b}),\lambda_{max}(\varGamma_{b})/\lambda_{min}(\varGamma_{a})\} \geq 1$, and next, we discuss the switching situation in two cases: 

Case a): if $ t \in  [t_{k},a_{k}) $, it follows from (\ref{VV6}) that 
\vspace{-2pt}
\begin{align}
V(t) & \leq e^{-\lambda_{a}(t-t_{k}) } V_{a}(t_{k}) \leq \mu e^{-\lambda_{a}(t-t_{k}) } V_{b}(t^{-}_{k}) \notag \\
& \leq \mu e^{-\lambda_{a}(t-t_{k}) } [ e^{\lambda_{b}(t_{k}-t_{k-1}) } V_{b}(t_{k-1}) ]  \notag \\ 
& \leq \mu e^{-\lambda_{a}(t-t_{k})} e^{\lambda_{b}(t_{k}-t_{k-1}) } [  \mu V_{a}(t^{-}_{k-1})  ] \notag \\ 
& = \mu^{2} e^{-\lambda_{a}(t-t_{k})} e^{\lambda_{b}(t_{k}-t_{k-1}) }   V_{a}(t^{-}_{k-1})  \notag \\ 
& \leq \mu^{2} e^{-\lambda_{a}(t-t_{k})} e^{\lambda_{b}(t_{k}-t_{k-1}) }  [ e^{-\lambda_{a}(t_{k-1}-t_{k-2})} V_{a}(t_{k-2})  ]   
\notag \\ 
& \leq \cdots  \leq \mu^{k+1} e^{-\lambda_{a} | \Sigma_{s} (t_{0}, t)|} e^{\lambda_{b} | \Sigma_{a} (t_{0}, t)| } V_{a}(t_{0}). \label{VV7}
\end{align}

Case b): if $ t \in   [a_{k},t_{k+1}) $, it follows from (\ref{VV6}) that 
\vspace{-2pt}
\begin{align}
V(t) & \leq e^{\lambda_{b}(t-a_{k}) } V_{b}(a_{k}) \leq \mu e^{\lambda_{b}(t-a_{k}) } V_{b}(a^{-}_{k}) \notag \\
& \leq \mu e^{\lambda_{b}(t-a_{k})} e^{-\lambda_{a}(a_{k}-a_{k-1}) } [  \mu V_{b}(a^{-}_{k-1})  ] \notag \\ 
& \leq \mu^{2} e^{\lambda_{a}(t-a_{k})} e^{-\lambda_{a}(a_{k}-a_{k-1}) }  [ e^{\lambda_{b}(a_{k-1}-a_{k-2})} V_{b}(a_{k-2})  ]   
\notag \\ 
& \leq \cdots \notag \\
& \leq \mu^{k+1} e^{-\lambda_{a} | \Sigma_{s} (t_{0}, t)|} e^{\lambda_{b} | \Sigma_{a} (t_{0}, t)| } V_{a}(t_{0}). \label{VV8} 
\end{align} 
 
iv) bounds on DoS attack frequency and duration

Notice that  $N_{a}(t_{0},t)=k+1$ for $t \in [t_{k},a_{k})$ or $t \in [a_{k},t_{k+1})$. Thus, for $\forall t\geq t_{0},$ it follows from (\ref{VV7}) and (\ref{VV8}) that 
\vspace*{-1pt}
\begin{equation}
V(t)\leq \mu^{N_{a}(t_{0},t)}e^{-\lambda_{a} |\Xi_{s}(t_{0},t)|}e^{\lambda_{b}|\Xi_{a}(t_{0},t)|}V(t_{0}).
\label{VV9}
\end{equation}

Notice that for all $t\geq t_{0},$ $|\Xi_{s}(t_{0},t)|=t-t_{0}-|\Xi_{a}(t_{0},t)|$ and $|\Xi_{a}(t_{0},t)|\leq T_{0}+(t-t_{0})/T_{a} $ by  Definition \ref{Attack Duration}. Thus, we have 
\vspace*{-2pt}
\begin{align}
&-\lambda_{a}(t-t_{0}-|\Xi_{a}(t_{0},t)|)+\lambda_{b}|\Xi_{a}(t_{0},t)|  \notag \\
&= -\lambda_{a}(t-t_{0})+(\lambda_{a}+\lambda_{b})  |\Xi_{a}(t_{0},t)|  \notag \\
&\leq -\lambda_{a}(t-t_{0})+(\lambda_{a}+\lambda_{b}) [T_{0}+(t-t_{0})/T_{a}].  \label{VV10}
\end{align}

Substituting (\ref{VV10}) into (\ref{VV9}) yields
\vspace*{-2pt}
\begin{align}
V(t) &\leq \mu ^{N_{a}(t_{0},t)}e^{-\lambda_{a}(t-t_{0}-| \Xi_{a}(t_{0},t)|)}e^{\lambda_{b}|\Xi_{a}(t_{0},t)|}V(t_{0})  \notag \\
&\leq e^{(\lambda_{a}+\lambda_{b})T_{0}}e^{-\lambda_{a}(t-t_{0})}e^{\frac{(\lambda_{a}+\lambda_{b})}{T_{a}}(t-t_{0})}
\notag \\
& \ \ \ \times e^{\ln (\mu ) N_{a}(t_{0},t)}V(t_{0}).  \label{VV11}
\end{align}

By exploiting the attack condition in (\ref{Condition1}), we can obtain that 
\vspace*{-3pt}
\begin{equation}
\ln(\mu) N_{a}(t_{0},t) \leq \ln(\mu) N_{0}+ \eta^{*}(t-t_{0}).   
\label{VV111}
\end{equation} 

Let $\eta=\lambda_{\alpha}-(\lambda _{a}+\lambda _{b})/T_{a}-\eta^{\ast}>0$. 
Based on another attack condition in (\ref{Condition2}), and using (\ref{VV111}), we further rewrite (\ref{VV11}) as
\vspace*{-2pt}
\begin{equation}
V(t)\leq e^{( \lambda_{a}+ \lambda_{b}) T_{0} + \ln(\mu) N_{0}} \ e^{-\eta (t-t_{0})}V(t_{0}).
\label{VV12}
\end{equation} 

Further, it follows from (\ref{VV1}), (\ref{VV2}), (\ref{VV6})  and (\ref{VV12}) that 
\vspace*{-3pt}
\begin{equation}
||\tilde{\chi}(t)||^{2} \leq \varsigma e^{-\eta
	(t-t_{0})} ||\tilde{\chi}(t_{0})||^{2},   
\label{VV13}
\end{equation} 
where $ \varsigma= e^{(\lambda_{a}+ \lambda_{b})T_{0} +2 \ln(\mu) } \varsigma_{a}/\varsigma_{b}$, $ \varsigma_{a}= \max\{\lambda_{max}(\varGamma_{a}),\lambda_{max}( \\ \varGamma_{b})\}$, and $ \varsigma_{b}= \min\{\lambda_{min}(\varGamma_{a}),\lambda_{min}(\varGamma_{b})\}$. 

Thus, it follows from (\ref{VV13}) that  $ \tilde{x}-X\tilde{\varrho} $, $ \tilde{\varrho}  + \tilde{z} $, and $ \tilde{y} $ are bounded, and converge to zero exponentially. That is,  $ \lim_{t\rightarrow \infty} (\tilde{x}-X\tilde{\varrho})=\textbf{0}$, $ \lim_{t\rightarrow \infty} (\tilde{\varrho}  + \tilde{z})=\textbf{0}$, and $ \lim_{t\rightarrow \infty} \tilde{y} =\textbf{0}$. It follows from  $ C_{i}X_{i}=I_{q} $ in (\ref{RegulationEquation}) that $ \lim_{t\rightarrow \infty} C(\tilde{x}-X\tilde{\varrho})=\lim_{t\rightarrow \infty}(\tilde{y}-\tilde{\varrho})=\textbf{0}$, i.e., $ \lim_{t\rightarrow \infty}\tilde{\varrho} =\textbf{0}  $. Thus, $ \lim_{t\rightarrow \infty}\tilde{z} =\textbf{0}  $. 
Then, based on $ \lim_{t\rightarrow \infty} \tilde{y} =\textbf{0}$, the problem in (\ref{P2}) is solvable. Hence, it can be concluded that $ y_{i} $ converges to $y^{*}$ exponentially for heterogeneous linear multi-agent systems under DoS attacks over random digraphs.    
\end{proof}

In the absence of attacks, Theorem 1 is reduced as follows.  

\vspace*{2pt} 
\textbf{Corollary 1:} 
In the absence of attacks, the optimal solution is achieved  under the following distributed algorithm, provided that $A_{i}-B_{i}K_{i}$ is Hurwitz and $ (U_{i}, W_{i},  X_{i})$ is a solution to (\ref{RegulationEquation}). 
\vspace{-3pt}
\begin{subequations}\label{Controller1}	
	\begin{align}
	u_{i}&=-K_{i}x_{i}-(U_{i}-K_{i}X_{i})\varrho_{i}+W_{i}\vartheta_{i},  \label{C1} \\
	\dot{\varrho}_{i}&=\vartheta_{i}=- \beta \sum_{j\in \mathcal{N}_{i}(t) }a_{ij}(t) (\varrho_{i}-\varrho_{j} + z_{i}-z_{j})   \notag \\
	& \ \ \ \ \ \ -\triangledown f_{i}(y_{i}) - \alpha \beta \sum_{j\in \mathcal{N}_{i}(t) }a_{ij}(t) (y_{i}-y_{j}), \label{C2} \\
	\dot{z}_{i} & = \alpha \beta \sum_{j\in \mathcal{N}_{i}(t) }a_{ij}(t) (y_{i}-y_{j}),  i \in \mathcal{V}. \label{C3}
	\end{align}
\end{subequations}
 
\begin{proof}
the proof is similar to Step (i), and is omitted.		
\end{proof}

\vspace*{-7pt}
\subsection{Distributed Event-Triggered Strategy Against DoS Attacks}
\vspace*{-1pt} 
The time-based scheme in the subsection A is implemented by fixed and periodic samplings, which require all agents to exchange information synchronously.
Let $ \{t^{i}_{k}\} $ denote a sequence of some transmission time instants with $ 0 = t^{i}_{0} < t^{i}_{1} < t^{i}_{1} \cdots <  t^{i}_{k}< \cdots $. Unlike known time-based transmission time instants $ t_{k} $, an event-triggered transmission attempt $ t_{k}^{i} $ will be scheduled to be resilient against DoS attacks. If $ t_{k}^{i} \in \Xi_{a}(t_{0},t)$, an  information transmission attempt suffers from the DoS attack, and thereby is unsuccessful. If $ t_{k}^{i} \in \Xi_{s}(t_{0},t)$, the information can be successfully transmitted. It is thus desirable to develop a new way to update the distributed optimization algorithm and determine the next triggering instants for information transmissions over insecure communication. The main challenges include twofolds: \textit{i) how to design a distributed dynamic event-triggered condition without Zeno behavior, and ii) how to illustrate the validity of distributed optimization algorithm to guarantee global exponential convergence.}  
 
\vspace{2pt} 
Now, propose the following distributed optimization algorithm to handle the effect of DoS attacks over random digraphs.
\vspace{-3pt}
\begin{subequations}\label{Controller3}	
\begin{align}
u_{i}&=-K_{i}x_{i}-(U_{i}-K_{i}X_{i})\varrho_{i}+W_{i}\vartheta_{i},  \label{TCC1} \\
\dot{\varrho}_{i}&=\vartheta_{i}= -\triangledown f_{i}(y_{i}) -\beta \bar{e}_{\varrho z i} - \alpha \beta \bar{e}_{y i} , \label{TCC2} \\
\dot{z}_{i} & = \alpha \beta \bar{e}_{y i},   t \in [t^{i}_{k}, t^{i}_{k+1}),  \ i \in \mathcal{V},  \label{TCC3} 
\end{align}
\end{subequations}
where the consensus errors under DoS attacks are described by
\vspace{-3pt} 
\begin{equation}
\bar{e}_{\varrho z i} = 
\left\{
\begin{array}{c} 
e_{\varrho z i}, \   \text{ if} \  t \in   [ t^{i}_{k}, t^{i}_{k+1})  \ \cap \ t=t^{i}_{k} \in \Xi_{s}(t_{0},t),   \\  
\hspace{-6.3em}
\textbf{0},  \ \  \text{ if} \ t=t^{i}_{k} \in   \Xi_{a}(t_{0},t),  
\end{array} 
\right.   \label{TCC4} 
\end{equation}
\vspace{-3pt}
\begin{equation}
\bar{e}_{y i} = 
\left\{
\begin{array}{c} 
e_{y i}, \ \ \ \text{ if} \   t \in [t^{i}_{k}, t^{i}_{k+1})  \ \cap   \ t=t^{i}_{k} \in \Xi_{s}(t_{0},t),  \\  
\hspace{-6.3em}
\textbf{0},  \ \ \ \text{ if} \ t=t^{i}_{k} \in   \Xi_{a}(t_{0},t),  
\end{array} 
\right. \label{TCC5}
\end{equation}
where $e_{\varrho z i}= \sum_{j\in \mathcal{N}_{i}(t) }a_{ij}(t) (\hat{\varrho}_{i}-\hat{\varrho}_{j} + \hat{z}_{i}-\hat{z}_{j}) $,  $ \hat{\varrho}_{i}=\varrho_{i}(t^{i}_{k}) $, $ \hat{z}_{i}=z_{i}(t^{i}_{k}) $, and $ e_{y i}=\sum_{j\in \mathcal{N}_{i}(t)} a_{ij}(t) (\hat{y}_{i}-\hat{y}_{j})$, $ \hat{y}_{i}=y_{i}(t^{i}_{k}) $.

\vspace{3pt} 
To specify event time instants, denote the measurement errors 
\vspace*{-3pt}
\begin{equation}
\tilde{e}_{yi}=\hat{y}_{i}(t)-y_{i}(t),  \ \tilde{e}_{\varrho z i}=\hat{\varrho}_{i}(t)+ \hat{z}_{i}(t)-(\varrho_{i}(t)+ z_{i}(t)). 
\label{MeasurementError}
\end{equation} 
 
Then, a dynamic event-triggering scheme is developed as 
\vspace{-1pt} 
\begin{equation}
\hspace{-0.3em}
t^{i}_{k+1}= 
\left\{
\begin{array}{c} 
\hspace{-0.5em}
\underset{t>t^{i}_{k}}{ \inf} \{t-t_{k}^{i} |  \sigma_{gi}
g_{i}(t) \leq \eta_{gi}(t) \ \text{or} \ \sigma_{hi} h_{i}(t) \leq \eta_{hi}(t) \}, \\  \text{ if}  \ t^{i}_{k} \in \Xi_{s}(t_{0},t),   \\   
\hspace{-6.6em}
t_{k}^{i}+\kappa^{i}_{k}, \ \ \ \ \ \ \ \ \  \text{ if}  \  t^{i}_{k} \in  \Xi_{a}(t_{0},t),  
\end{array} 
\right.  \label{EV1} 
\end{equation}
where  $\kappa^{i}_{k}>0 $ is a dwell time to be determined, $ \sigma_{gi}>0 $, $ \sigma_{hi}>0 $, and the triggering functions $ g_{i}(t) $, $ h_{i}(t) $ are given by
\vspace{-3pt} 
\begin{equation} \label{triggeringfunction} 
g_{i}(t)=\tilde{e}_{y i}^{2} - \theta_{gi}e_{yi}^{2},
\ h_{i}(t)= \tilde{e}^{2} _{\varrho zi}- \theta_{hi} e^{2}_{\varrho z i},  
\end{equation}
where  $ \theta_{gi}$, $\theta_{hi} \in [0,1) $, and $ \eta_{gi}(t)$, $\eta_{hi}(t)$ are auxiliary variables satisfying for $ \eta_{gi}(t_{0})$, $ \eta_{hi}(t_{0})$, $ k_{gi}, k_{hi} >0$, and $ \delta_{gi}, \delta_{hi} \in [0,1) $,
\vspace{-5pt} 
\begin{equation}
\dot{\eta}_{gi} = 
\left\{
\begin{array}{c} 
\hspace{-0.6em}
-k_{gi}\eta_{gi}-\delta_{gi}g_{i}(t),    
  \ \  \text{ if}  \  t^{i}_{k}  \in   \Xi_{s}(t_{0},t),  \\ 
\ \ \ \ \ \ \ 0 , \ \  \ \ \ \ \ \ \ \ \ \ \ \   \text{ if}  \  t^{i}_{k} \in  \Xi_{a}(t_{0},t),
\end{array} 
\right. 
 \label{EV3} 
\end{equation}
\begin{equation}
\dot{\eta}_{hi} = 
\left\{
\begin{array}{c} 
\hspace{-0.6em}
-k_{hi}\eta_{hi}-\delta_{hi}h_{i}(t),    
\ \  \text{ if}  \  t^{i}_{k}  \in   \Xi_{s}(t_{0},t),  \\ 
\ \ \ \ \ \ \ 0 , \ \  \ \ \ \ \ \ \ \ \ \ \ \ \   \text{ if}  \   t^{i}_{k} \in   \Xi_{a}(t_{0},t). 
\end{array} 
\right. 
\label{EV4} 
\end{equation}
 
\vspace{-2pt} 
Next, we firstly introduce the following lemma. 

\begin{lemma} \label{lemma4}
Under the dynamical event-triggered scheme (\ref{EV1}), the term $ \eta_{i}(t)$ below is always positive for $ \forall \eta_{gi}(t_{0}), \eta_{hi}(t_{0})>0  $,   
\vspace{-9pt}
\begin{equation}
\eta_{i}(t)= \eta_{gi}(t)+\eta_{hi}(t)>0, \  \forall i \in \mathcal{V}. \label{EV6}
\end{equation}
\end{lemma}
 
\vspace{-3pt} 
\begin{proof}
under the proposed event-triggered scheme  (\ref{EV1})-(\ref{EV4}), the transmission attempts $\{t^{i}_{k}\}_{k \in \mathrm{N}^{+}}= \{t^{i}_{k_{s}}\}_{s \in \mathrm{N}^{+}}  \bigcup \ \{t^{i}_{k_{a}}\}_{a \in \mathrm{N}^{+}} $ can be generated, where $ \{t^{i}_{k_{s}}\}_{s \in \mathrm{N}^{+}} $ denotes the successful transmission attempts while $ \{t^{i}_{k_{a}}\}_{a \in \mathrm{N}^{+}} $ denotes the unsuccessful transmission attempts. Next, two cases are considered: 

Case i): choose any $ t^{i}_{k_{s_{0}}} \in \Xi_{s}(t_{0},\infty) $  and let $ a_{1}=\min_{a} \{ t^{i}_{k_{a}} >t^{i}_{k_{s_{0}}} \} $. Then, when $ t \in [t^{i}_{k}, t^{i}_{k+1}) $,  $ \tilde{e}^{2}_{y i}(t) - \theta_{gi}  e^{2}_{yi}\leq \eta_{gi}(t)/\sigma_{gi}. $
Hence, we can obtain from (\ref{EV3}) that 
\vspace{-3pt} 
\begin{equation}
\dot{\eta}_{gi}(t) = -k_{gi}\eta_{gi}(t) -  \delta_{gi}g_{i}(t) \geq - (k_{gi} + \delta_{gi}/ \sigma_{gi} ) \eta_{gi}(t), \label{EV7}
\end{equation}
which implies that based on the mathematical induction, we have that for $ t \in [t^{i}_{k}, t^{i}_{k+1}) \in [t^{i}_{k_{s_{0}}}, t^{i}_{k_{a_{1}}})$,  
\vspace{-2pt} 
\begin{align}
\eta_{gi}(t)  &\geq \eta_{gi}(t^{i}_{k}) e ^{-  (k_{gi} + \delta_{gi}/ \sigma_{gi} ) (t-t^{i}_{k}) } \geq  \cdots  \notag \\ 
&\geq \eta_{gi}(t^{i}_{k_{s_{0}}}) e ^{-  (k_{gi} + \delta_{gi}/ \sigma_{gi} ) (t-t^{i}_{k_{s_{0}}}) } >0. \label{EV8}
\end{align}

\vspace{-3pt}
Case ii): denote $ s_{1}=\min_{s} \{ t^{i}_{k_{s}} >t^{i}_{k_{a_{1}}} \} $, and we have that for $ t \in [t^{i}_{k}, t^{i}_{k+1}) \in [t^{i}_{k_{a_{1}}}, t^{i}_{k_{s_{1}}})$, it follows from (\ref{EV3}) that  $ \dot{\eta}_{gi}(t) =0 $. Thus, $ \eta_{gi}(t) $ is time-invariant and since $ \eta_{gi}(t_{0}) >0$, we have 
\vspace{-3pt} 
\begin{equation}
\eta_{gi}(t) = \eta_{gi}(t^{i}_{k}) =\cdots= \eta_{gi}(t^{i}_{k_{a_{1}}})>0. \label{EV9}
\end{equation}

\vspace{-3pt} 
Then, we obtain from both cases that $ \eta_{gi}(t)>0, \forall i \in \mathcal{V} $.

\vspace{1pt} 
Similarly, it can be derived from (\ref{EV4}) that $ \eta_{hi}(t)>0, \forall i \in \mathcal{V} $ for both cases by following the same procedure.
\end{proof}

\vspace{2pt}
Substituting (\ref{Controller3})-(\ref{TCC5}) into (\ref{Dynamics}) yields the following system 
\vspace{-3pt}
\begin{equation} \label{Controller3Compact}
\dot{\chi} = 
\left\{ 
\begin{array}{l}
\left[ \begin{array}{ccc}
(A-BK)x+BW \vartheta-B(U-KX)\varrho  \\
-\triangledown \tilde{f}(y) - \alpha \beta \tilde{\mathcal{L}}(t) \hat{y} - \beta \tilde{\mathcal{L}}(t) (\hat{\varrho} + \hat{z}) \\
\alpha \beta  \tilde{\mathcal{L}}(t) \hat{y}
\end{array} \right],   \\  \text{ if} \  t \in  [ t^{i}_{k}, t^{i}_{k+1})   \cap  t=t^{i}_{k} \in \Xi_{s}(t_{0},t),   \\
\left[ \begin{array}{ccc}
(A-BK)x+BW \vartheta-B(U-KX)\varrho  \\
-\triangledown \tilde{f}(y)  \\
\textbf{0}
\end{array} \right],   \\  \text{ if}  \  t=t^{i}_{k} \in \Xi_{a}(t_{0},t), k=0,1,2\cdots,
\end{array}
\right. 	
\end{equation} 
where $ \hat{y}, \hat{\varrho} , \hat{z} $ are the stack vectors of $ \hat{y}_{i}, \hat{\varrho}_{i} , \hat{z}_{i} $, respectively, and 
\hspace{-5pt} 
\begin{equation}
\hat{y}_{i}(t) = y_{i}+\tilde{e}_{yi}, \ \hat{\varrho}_{i}(t) = \varrho_{i}+\tilde{e}_{\varrho i}, \ \hat{z}_{i}(t) = z_{i}+\tilde{e}_{zi}. \label{EV10}
\end{equation}

  
\begin{theorem} \label{Theorem3}
Given Assumptions \ref{Assumption2}-\ref{Assumption4}, Problem 1 is solvable for any $ x_{i}(0)$, $\varrho_{i}(0) $ and $z_{i}(0) $ under the proposed resilient distributed optimization algorithm in (\ref{Controller3})-(\ref{TCC5}) together with the distributed event-triggered scheme in (\ref{EV1})-(\ref{EV4}), provided that $A_{i}-B_{i}K_{i}$ is Hurwitz, $ U_{i}, W_{i},  X_{i}$ are solutions to (\ref{RegulationEquation}), and 
\vspace*{-3pt}
\begin{equation}
T_{f} >(\ln(\tilde{\mu})+(\tilde{\lambda}_{a}+\tilde{\lambda}_{b})\kappa_{*})/\tilde{\eta}^{\ast}, \ T_{a}> (\tilde{\lambda}_{a}+\tilde{\lambda}_{b}) /(\tilde{\lambda}_{a}-\tilde{\eta}^{\ast }),
\label{Condition11}
\end{equation}
where $\tilde{\lambda}_{a}, \tilde{\lambda}_{b} , \tilde{u} , \kappa_{*}$ are  positive scalars to be determined later.
\end{theorem}

\begin{proof}
the proof includes the following steps:  

\vspace*{1pt}
\textbf{Step 1} (two intervals classification): two time intervals where (\ref{EV1})-(\ref{EV4}) holds and does not hold are characterized. Consider sequences: $\{t_{k}^{i}\}_{k \in \mathbb{N}^{+}}$ and  $\{a_{k}\}_{k\in \mathbb{N}^{+}}.$ The following $\mathcal{F}:=\{(i,k)\in \mathcal{V\times}\mathbb{N}^{+} \text{ }|t_{k}^{i}\in \cup _{k\in \mathbb{N}^{+}} \mathcal{A}_{k}\}$ is a set of integers related to an updating attempt occurring under attacks. 
Due to the finite sampling rate, a time interval will necessarily elapse from $t^{i}_{k}+a_{k}$ to the time when agents successfully sample and transmit. It is upper bounded by $\sup_{(i,k)\in \mathcal{F}} \kappa^{i}_{k} \leq \kappa_{*}$. Hence, a DoS free interval of a length greater than $\kappa_{*}$ ensures that each agent can sample and transmit. The $m$-th time interval where (\ref{EV1})-(\ref{EV4}) do not need to hold is 
\vspace*{-4pt}
\begin{equation}
\mathfrak{A}_{m}=[a_{m},a_{m}+\tau_{m}+\kappa_{*}).
\label{e_i does not need to hold}
\end{equation}

\vspace*{-2pt}
Thus, the time interval $[\tau ,t)$ consists of the following two union of sub-intervals: $[\tau ,t)=\tilde{\Xi}_{s}(\tau ,t)\cup \tilde{\Xi}_{a}(\tau,t)$ with
\vspace*{-2pt}
\begin{eqnarray}
\hspace{1.2em}
\tilde{\Xi}_{a}(\tau ,t) :=\cup \ \mathfrak{A}_{m}  \cap  [\tau ,t],  \  \tilde{\Xi}_{s}(\tau ,t):=[\tau ,t]\backslash   \tilde{\Xi}_{a}(\tau,t).
\label{Success interval}
\end{eqnarray}

\textbf{Step 2} (Lyapunov stability analysis) \\
i) consider the time interval $ \tilde{\Xi}_{s}(\tau ,t)$ over which (\ref{EV1})-(\ref{EV4}) hold. Denote  $ \tilde{\varGamma}_{a}=\text{diag}\{2 \tilde{P}_{a},  \Pi \otimes I_{q}, \Pi \otimes I_{q} \} $ and then select 
\vspace{-6pt}
\begin{equation}
W_{a}=V_{a}+V_{\eta}=\frac{1}{2} \tilde{\chi}^{T} \tilde{\varGamma}_{a} \tilde{\chi} + \sum_{i=1}^{N} \eta_{i}(t).
 \label{EEV1}
\end{equation}   

\vspace{-3pt}
Then, the time derivative of $ W_{a} $ along (\ref{Controller3Compact})-(\ref{EV10}) is given by 
\vspace{-1pt}
\begin{align}
\dot{W}_{a}
&= \tilde{y}^{T} \tilde{\Pi} C\tilde{A} (\tilde{x}-X\tilde{\varrho}) - (\tilde{y}+\tilde{\varrho}  + \tilde{z})^{T} \tilde{\Pi}  F + \tilde{y}^{T} \tilde{\Pi} (  - \alpha \beta \tilde{\mathcal{L}}(t) \tilde{y} \notag  \\
&   - \beta \tilde{\mathcal{L}}(t) (\tilde{\varrho}  + \tilde{z})) - (\tilde{x}-X\tilde{\varrho})^{T} O  (\tilde{x}-X\tilde{\varrho})   -\beta\tilde{y}^{T} \tilde{\Pi} \tilde{\mathcal{L}}(t) \notag \\
& \times (\alpha \tilde{e}_{y}  + \tilde{e}_{\varrho z}) -\beta (\tilde{\varrho}  + \tilde{z})^{T} [ Q(t) (\tilde{\varrho}  + \tilde{z}) + \tilde{\Pi} \tilde{\mathcal{L}}(t) \tilde{e}_{\varrho z}] + \dot{V}_{\eta} \notag \\  
&\leq  - [ \epsilon_{0} - \frac{||\tilde{A}||^{2}}{4 \epsilon_{1}}  ]   (\tilde{x}-X\tilde{\varrho})^{T}(\tilde{x}-X\tilde{\varrho}) + \epsilon_{1}||\tilde{\Pi} C ||^{2}  \tilde{y}^{T}\tilde{y}  \notag \\
&  -\iota_{min} \pi_{min} \tilde{y}^{T}  \tilde{y} -  \alpha \beta  \tilde{y}^{T} Q(t) \tilde{y} -  \beta  \tilde{y}^{T}\tilde{\Pi} \tilde{\mathcal{L}}(t) (\tilde{\varrho}  + \tilde{z}) \notag \\
&  + \frac{l^{2}_{max} \pi^{2}_{max}}{4 \epsilon_{2}} ||\tilde{y}||^{2}+ \epsilon_{2} ||\tilde{\varrho}  + \tilde{z}||^{2} -\beta (\tilde{\varrho}  + \tilde{z})^{T} Q(t) (\tilde{\varrho}  + \tilde{z}) \notag \\
&+ \frac{1}{4c_{1}} ||\alpha\beta  \tilde{\Pi} \tilde{\mathcal{L}}(t)||^{2}  ||\tilde{y}||^{2}  + c_{1} ||\tilde{e}_{y}||^{2}  +  \frac{1}{4c_{2}} || \beta  \tilde{\Pi} \tilde{\mathcal{L}}(t)||^{2} ||\tilde{y}||^{2} \notag \\
&  + c_{2} ||\tilde{e}_{\varrho z}  ||^{2} +  \frac{1}{4c_{2}} || \beta  \tilde{\Pi} \tilde{\mathcal{L}}(t)||^{2} ||\tilde{\varrho} + \tilde{z}||^{2} + c_{2} ||\tilde{e}_{\varrho z} ||^{2} \notag \\   
& - \sum_{i=1}^{N} k_{gi}\eta_{gi} -\sum_{i=1}^{N}\delta_{gi}   (\tilde{e}^{2}_{yi} - \theta_{gi} e^{2}_{yi} )  + \sum_{i=1}^{N} \theta_{gi} (e^{2}_{yi} - e^{2}_{yi}) \notag \\
& - \sum_{i=1}^{N} k_{hi}\eta_{hi} -\sum_{i=1}^{N}\delta_{hi}   (\tilde{e}_{\varrho zi}^{2} - \theta_{hi}e_{\varrho zi}^{2} )   + \sum_{i=1}^{N}  \theta_{hi} (e_{\varrho  zi}^{2} -e_{\varrho zi}^{2}) \notag \\
& \leq - \left[ \begin{array}{ccc}
\tilde{x}-X\tilde{\varrho}  \\
\tilde{\varrho}  + \tilde{z}   \\ 
\tilde{y}    
\end{array} \right] ^{T} \tilde{\varOmega}_{a} (t) \left[ \begin{array}{ccc}
\tilde{x}-X\tilde{\varrho}  \\
\tilde{\varrho}  + \tilde{z}   \\ 
\tilde{y}    
\end{array} \right] - \sum_{i=1}^{N} k_{gi}\eta_{gi} \notag \\
& \ \ \ + \sum_{i=1}^{N} (1-\delta_{gi})   (\tilde{e}^{2}_{yi} - \theta_{gi} e^{2}_{yi} )  - \sum_{i=1}^{N} k_{hi}\eta_{hi}   \notag \\
& \ \ \ + \sum_{i=1}^{N} (1-\delta_{hi}) [\tilde{e}_{\varrho zi}^{2} - \theta_{hi} e_{\varrho zi}^{2}]   \notag \\
& \leq - \tilde{\chi}^{T} \tilde{\varOmega}_{a} (t) \tilde{\chi} -  \sum_{i=1}^{N} k_{ci} \eta_{i}  \leq - \tilde{\lambda}_{a} W_{a},   \label{EEV2}
\end{align} 
where $ \tilde{\Pi}=\Pi  \otimes I_{q} $,  $ \tilde{\lambda}_{a}=\min \{ \lambda_{min} (\tilde{\varOmega}_{a} ), \underline{k}_{c} \} $, $ \underline{k}_{c}=\min  \{k_{ci} \}$, 
$ k_{ci}=\min\{k_{gi}-\frac{1-\delta_{gi}}{\sigma_{gi}} , k_{hi}-\frac{1-\delta_{hi}}{\sigma_{hi}} \} >0$,  $ c_{1}=1-c_{3}\overline{\theta}_{g} \|\tilde{L}\|^{2}>0$, 

$ c_{2}=(1-c_{3}\overline{\theta}_{h} \|\tilde{L}\|^{2})/2 >0$, $ \overline{\theta}_{g}=\max  \{\theta_{gi}\}$, $ \overline{\theta}_{h}=\max \{\theta_{hi}\}$, $ \tilde{\varOmega}_{a} $ is similar to $\varOmega_{a} $ with $ \tilde{\varOmega}^{2}_{a} =\varOmega^{2}_{a} - (\frac{1}{4c_{2}} \|\beta \tilde{\Pi} \tilde{L} \|^{2}+\frac{1}{4c_{3}}\|\tilde{L}\|^{2})\textbf{I} $, $ \tilde{\varOmega}^{3}_{a} =\varOmega^{3}_{a} -((\frac{1}{4c_{1}}+ \frac{1}{4c_{2}})\|\alpha \beta \tilde{\Pi} \tilde{L} \|^{2}+\frac{1}{4c_{3}}\|\tilde{L}\|^{2}) \textbf{I} $.  

\vspace{2pt}
ii) consider the time interval $ \tilde{\Xi}_{a}(\tau ,t)$ over which (\ref{EV1})-(\ref{EV4}) do not necessarily hold. 
Let $ \tilde{P}_{b}>0 $, $ \tilde{S}=\bar{S} \otimes I_{q} $ with $ \bar{S}> $ being a diagonal matrix and choose the Lyapunov function candidate as
\vspace{-5pt} 
\begin{equation}
W_{b}=\frac{1}{2} \tilde{\chi}^{T} \tilde{\varGamma}_{b}\tilde{\chi}+ \sum_{i=1}^{N}\eta_{i}(t), \  \tilde{\varGamma}_{b} = \left[ \begin{array}{ccc}
2\tilde{P}_{b} & \textbf{0} & \textbf{0} \\
* & \tilde{S}  & \textbf{0} \\
* & *  & \tilde{S} 
\end{array} \right]. \label{EEV3} 
\end{equation} 

Similar to (\ref{VV3}), the time derivative of $ W_{b} $ along (\ref{Controller3Compact}) is   
\vspace{-5pt}
\begin{equation}
\dot{W}_{b} \leq  \tilde{\chi}^{T} \tilde{\varOmega}_{b} \tilde{\chi}  \leq   \tilde{\lambda}_{b} W_{b} ,  
\label{EEV4}  
\end{equation}
where $ \lambda_{b}= \max \{ \sigma_{max} (\tilde{\varOmega}_{b} ), 1 \} >0 $, and $ \tilde{\varOmega}_{b} $ is similarly defined as (\ref{VV4}) with
$ \tilde{\epsilon}_{3}=\tilde{\epsilon}_{1}||(\tilde{S} \otimes I_{q} ) C ||^{2} + \frac{l^{2}_{max}\pi^{2}_{max}}{4 \tilde{\epsilon}_{2}} -\iota_{min} \lambda_{min}(\tilde{S})  $. 

\vspace{3pt}
iii) convergence analysis of a piecewise Lyapunov function.

\vspace{2pt}
Similar to (\ref{VV6})-(\ref{VV8}), let $ W(t)=W_{\tilde{\delta}(t)} $ and $ \tilde{\delta}(t) \in \{a, b\} $, where $ W_{a} $ and $W_{b}$ are defined in (\ref{EEV1}) and (\ref{EEV3}), respectively. Suppose that $ W_{a} $ is activated in $[a_{m-1}+\tau_{m-1},a_{m})$ and $ W_{b} $ is activated in $[a_{m},a_{m}+\tau_{m}+\kappa_{*})$. Then, by the Comparison lemma, we get
\vspace{-4pt}
\begin{equation}
W(t)=\left\{ 
\begin{array}{c}
\hspace{-0.3em}  
e^{-\tilde{\lambda}_{a}(t-t_{m-1}-\tau_{m-1}) } W_{a}(t_{m-1}+\tau_{m-1}) ,   \\ 
e^{\tilde{\lambda}_{b}(t-a_{m}) } W_{b}(a_{m}).
\end{array}%
\right. 
\label{EEV5} 
\end{equation}

\vspace{-2pt}
Denote $ \tilde{\mu} = \max\{ \lambda_{max}(\tilde{\varGamma}_{a})/\lambda_{min}(\tilde{\varGamma}_{b}),\lambda_{max}(\tilde{\varGamma}_{b})/\lambda_{min}(  \tilde{\varGamma}_{a}) \} \geq 1$.
Next, we discuss the situation in two cases: 

\vspace{3pt}
Case a): if $ t \in  [a_{m-1}+\tau_{m-1},a_{m}) $, it follows from (\ref{EEV5}) that 
\vspace{-5pt}
\begin{align}
W(t) & \leq e^{-\tilde{\lambda}_{a}(t-a_{m-1}-\tau_{m-1}) } W_{a}(a_{m-1}+\tau_{m-1})  \notag \\
& \leq \tilde{\mu} e^{-\tilde{\lambda}_{a}(t-a_{m-1}-\tau_{m-1})}    W_{b}(a^{-}_{m-1}+\tau^{-}_{m-1})  
\leq \cdots \notag \\
& \leq \tilde{\mu}^{m} e^{-\tilde{\lambda}_{a} | \tilde{\Sigma}_{s} (t_{0}, t)|} e^{\tilde{\lambda}_{b} |\tilde{ \Sigma}_{a} (t_{0}, t)| } W_{a}(t_{0}). \label{EEV6}
\end{align}

\vspace{-3pt}
Case b): if $ t \in  [a_{m},a_{m}+\tau_{m}+\kappa_{*}) $,  similarly  
\vspace{-5pt}
\begin{align}
W(t) & \leq e^{\tilde{\lambda}_{b}(t-a_{m}) } W_{b}(a_{m}) \leq \tilde{\mu} e^{\tilde{\lambda}_{b}(t-a_{m}) } W_{a}(a^{-}_{m}) \notag \\
& \leq e^{\tilde{\lambda}_{b}(t-a_{m}) } [ e^{-\tilde{\lambda}_{a}(a_{m}-a_{m-1}- \tau_{m-1}) } W_{a}(a_{m-1}+\tau_{m-1}) ]  \notag \\  
& \leq \cdots  \leq \tilde{\mu}^{m+1} e^{-\tilde{\lambda}_{a} | \tilde{\Sigma}_{s} (t_{0}, t)|} e^{\tilde{\lambda}_{b} | \tilde{\Sigma}_{a} (t_{0}, t)| } W_{a}(t_{0}). \label{EEV7} 
\end{align} 

\vspace{-2pt}
iv) bounds on the attack frequency and duration.

\vspace*{2pt}
From Definition \ref{Attack Frequency}, $N_{a}(t_{0},t)=m$ for $t \in  [a_{m-1}+\tau_{m-1},a_{m})$ and $m+1$ for $t \in   [a_{m},a_{m}+\tau_{m}+\kappa_{*})$. Thus, for $\forall t\geq t_{0}$, 
\vspace*{-2pt}
\begin{equation}
W(t)\leq \mu^{N_{a}(t_{0},t)}e^{-\tilde{\lambda}_{a} |\tilde{\Xi}_{s}(t_{0},t)|}e^{\tilde{\lambda}_{b}|\tilde{\Xi}_{a}(t_{0},t)|}W(t_{0}).
\label{EEV9}
\end{equation}

\vspace*{-2pt}
Notice that for all $t\geq t_{0},$ $|\tilde{\Xi}_{s}(t_{0},t)|=t-t_{0}-|\tilde{\Xi}_{a}(t_{0},t)|$ and $|\tilde{\Xi}_{a}(t_{0},t)| \leq |\Xi_{a}(t_{0},t)|+(1+N_{a}(t_{0},t))\kappa_{*} $. By Definition \ref{Attack Duration}, we have:  $- \tilde{\lambda}_{a}(t-t_{0}-|\tilde{\Xi}_{a}(t_{0},t)|)+\tilde{\lambda}_{b}|\tilde{\Xi}_{a}(t_{0},t)|= -\tilde{\lambda}_{a}(t-t_{0})+(\tilde{\lambda}_{a}+\tilde{\lambda}_{b}) [T_{0}+(t-t_{0})/T_{a}]$. Thus, it follows from (\ref{EEV9}) that 
 \vspace*{-2pt}
\begin{align}
W(t) & \leq  e^{(\tilde{\lambda}_{a}+\tilde{\lambda}_{b})T_{0}}e^{-\tilde{\lambda}_{a}(t-t_{0})}e^{\frac{(\tilde{\lambda}_{a}+\tilde{\lambda}_{b})}{\tau_{a}}(t-t_{0})}
\notag \\
& \ \ \ \times e^{[\ln (\tilde{\mu})+(\tilde{\lambda}_{a}+\tilde{\lambda}_{b}) N_{a}(t_{0},t)]}W(t_{0}).  \label{VVV11}
\end{align}

\vspace*{-3pt}
From (\ref{Condition11}), let $\tilde{\eta}=\tilde{\lambda}_{a}-(\tilde{\lambda}_{a}+\tilde{\lambda}_{b})/\tau _{a}-\tilde{\eta}^{\ast}>0$ so that
\vspace*{-2pt}
\begin{equation}
W(t)\leq e^{( \tilde{\lambda}_{a}+ \tilde{\lambda}_{b})(T_{0}+\kappa_{*}) +\tilde{u}N_{0}}e^{-\tilde{\eta}
	(t-t_{0})}W(t_{0}).  
\label{VVV12}
\end{equation} 

Thus, for certain scalar $ \tilde{\varsigma}>0 $, we can have 
\vspace*{-2pt}
\begin{equation}
\left\| \left[ \begin{array}{cc}
 \tilde{\chi}(t)   \\
 \eta (t)  
 \end{array} \right] \right\| ^{2} \leq \tilde{\varsigma} e^{-\tilde{\eta}
	(t-t_{0})}   \left\|\left[ \begin{array}{cc}
\tilde{\chi}(t_{0})   \\
\eta (t_{0})  
\end{array} \right] \right\| ^{2}.   
\label{VVV13}
\end{equation} 

Hence, we obtain that $ y_{i} $ converges to $y^{*}$ exponentially.  	
\end{proof}

\begin{theorem} \label{Theorem4}
The event-triggering instants determined by (\ref{EV1})-(\ref{EV4}) guarantee that no agents will exhibit Zeno behavior.
\end{theorem}

\begin{proof}
suppose that there exists Zeno behavior. That is, there exists an agent so that at certain time $ \mathcal{T}_{0} $, $ \lim_{t \rightarrow \infty} t^{i}_{k} = \mathcal{T}_{0}$. From the property of limits, we have that for certain  $ \varepsilon_{0}>0 $, there exists a positive integer $ \mathcal{N}(\varepsilon_{0}) $ such that 
\vspace*{-3pt}
\begin{equation}
t^{i}_{k} \in  [\mathcal{T}_{0}-\varepsilon_{0}, \mathcal{T}_{0}+\varepsilon_{0}], \ \forall k \geq \mathcal{N}(\varepsilon_{0}), 
\label{Z1}
\end{equation}
which implies that $ t^{i}_{\mathcal{N}(\varepsilon_{0})+1}-t^{i}_{\mathcal{N}(\varepsilon_{0})} < 2\varepsilon_{0} $.  

\vspace*{2pt} 
Based on (\ref{Dynamics}), (\ref{Controller3}) and (\ref{EV10}), the upper right-hand Dini derivative of $ \| \tilde{e}_{yi} \|  $ and $ \|  \tilde{e}_{\varrho zi} \| $ can be derived as 
\vspace*{-3pt}
\begin{subequations} \label{Z2}
\begin{align}
D^{+} \| \tilde{e}_{yi} \|  & = \| - \dot{y}_{i} \|  \leq \| C_{i} \tilde{A}_{i} (x_{i}- X_{i} \varrho_{i} ) + \dot{\varrho}_{i} \|,   \label{ED1} \\
 D^{+} \| \tilde{e}_{\varrho zi} \|   & = \|- \dot{\varrho}_{i} -\dot{z}_{i} \|  \leq    \| \dot{\varrho}_{i} + \dot{z}_{i} \|, \label{ED11} \\
\dot{\varrho}_{i}& = -\triangledown f_{i}(y_{i}) -\beta \bar{e}_{\varrho z i} - \alpha \beta \bar{e}_{y i} , \label{ED2} \\
\dot{z}_{i} & = \alpha \beta \bar{e}_{y i}, \ t \in [t^{i}_{k}, t^{i}_{k+1}).   \label{ED3} 
\end{align}
\end{subequations}


Case a): for  $ t \in  [ t^{i}_{k}, t^{i}_{k+1})   \cap   t=t^{i}_{k} \in \Xi_{s}(t_{0},t) $, we get  
\vspace*{-3pt} 
\begin{equation}
 D^{+} \| \tilde{e}_{yi} \|   \leq  \| C_{i} \tilde{A}_{i} (x_{i} - X_{i} \varrho_{i} ) -\triangledown f_{i}(y_{i}) - \beta e_{\varrho z i}  - \alpha \beta e_{y i}  \|. \label{Z3} 
\end{equation} 

From Theorem 2, we have that all the states $ \tilde{x}-X\tilde{\varrho} $, $ \tilde{\varrho}  + \tilde{z} $, $ \tilde{y} $ are bounded. Thus, for certain scalar $C_{0}>0$, 
\vspace*{-3pt} 
\begin{equation}
D^{+} \| \tilde{e}_{yi} \|   \leq  C_{0} . \label{Z4} 
\end{equation}

\vspace*{-3pt}
Since an event is triggered if (\ref{EV1}) is satisfied and $ || \tilde{e}_{yi} (t^{i}_{k})||=0 $, we have $ || \tilde{e}_{yi}(t) ||^{2} \geq   \theta_{gi} e^{2}_{yi} +  \eta_{gi}(t)/\sigma_{gi}  \geq \eta_{gi}(t)/\sigma_{gi} $ . Thus, 
\vspace*{-3pt}
\begin{equation}
|| \tilde{e}_{yi}(t) || \geq \sqrt{\eta_{gi}(t)/\sigma_{gi}} \geq \sqrt{\eta_{gi}(0)/\sigma_{gi}} e^{-\frac{k_{gi}+\delta_{gi}/\sigma_{gi}}{2} t}.
\label{Z5} 
\end{equation}

Hence, it follows from (\ref{Z4}) and (\ref{Z5}) that 
\vspace*{-3pt}
\begin{equation}
t^{i}_{\mathcal{N}(\varepsilon_{0})+1}-t^{i}_{\mathcal{N}(\varepsilon_{0})} \geq \frac{1}{C_{0}} \sqrt{\eta_{gi}(0)/\sigma_{gi}} e^{-\frac{k_{gi}+\delta_{gi}/\sigma_{gi}}{2} t^{i}_{\mathcal{N}(\varepsilon_{0})+1}}.   \notag 
\end{equation}

\vspace*{-3pt}
Select $ \varepsilon_{0} $ as the solution of $ \frac{1}{C_{0}} \sqrt{\eta_{gi}(0)/\sigma_{gi}} e^{-\frac{k_{gi}+\delta_{gi}/\sigma_{gi}}{2} \mathcal{T}_{0}} =2 \varepsilon_{0} e^{\frac{k_{gi}+\delta_{gi}/\sigma_{gi}}{2} \varepsilon_{0}} $. Then, we obtain that 
\vspace*{-3pt}
\begin{equation}
t^{i}_{\mathcal{N}(\varepsilon_{0})+1}-t^{i}_{\mathcal{N}(\varepsilon_{0})} \geq \frac{1}{C_{0}} \sqrt{\eta_{gi}(0)/\sigma_{gi}} e^{-\frac{k_{gi}+\delta_{gi}/\sigma_{gi}}{2} ( \mathcal{T}_{0} + \varepsilon_{0})}=2\varepsilon_{0} , \notag  
\end{equation}
which contracts (\ref{Z1}). Thus, no Zero behavior exists for Case a). 

\vspace*{4pt}
Case b): for $ t \in   [ t^{i}_{k}, t^{i}_{k+1})    \cap  t=t^{i}_{k} \in \Xi_{s}(t_{0},t) $,  $ D^{+} \| \tilde{e}_{\varrho zi} \|   \leq  \|  -\triangledown f_{i}(y_{i}) - \beta e_{\varrho z i}  \| $. Following a similar procedure, we can also show that no Zero behavior exists for $ \tilde{e}_{\varrho zi} $.

Overall, Zero behavior can be avoided for agent $ i $.
\end{proof}

\vspace*{3pt} 
\begin{remark}
Unlike traditional time-triggered control schemes in (\ref{Controller2}), an event-based transmission strategy is presented in (\ref{Controller3}) to save limited communication resources and to be resilient against attacks. The defined $ \{t^{i}_{k} \} $ is an aperiodic transmission sequence determined by (\ref{EV1}) for agent i. This means that each agent has its own triggering rule, that is, they transmit data in an asynchronous manner. In addition, the triggering process of each agent does not affect each other, that is, an agent and its neighbors are considered independent from the perspective of activating the triggering rules. Each agent needs to measure its own  information and has access to its neighboring information, which essentially reflects the basic requirement of distributed control strategy. 
\end{remark}

\begin{remark}
Inspired by \cite{Girard15TAC} and \cite{Yi17CDC }, a new dynamical event-triggered communication scheme in (\ref{EV1})-(\ref{EV4}) is designed over an insecure and unreliable network. For linear multi-agent systems, the proposed scheme in (\ref{EV1})-(\ref{EV4}) generates transmission updates under DoS attacks. General speaking, there exist several methods to exclude Zeno behavior, e.g., a pre-defined dwell time design in \cite{Hu19TCST} or combination with sample-data schemes in \cite{Dixion17TAC}. In contrast to the above literature, the dynamical event-triggered communication scheme in (\ref{EV1})-(\ref{EV4}) is developed to avoid the continuous communication and involvement of global graph information. In addition, its effectiveness has been verified even in the presence of DoS attacks over random digraphs.  
\end{remark}

\vspace{-3pt} 	
\section{Numerical Simulations}
\vspace{-1pt}
In this section, numerical simulation examples are presented to verify the effectiveness of the proposed designs. 

\vspace{2pt}
Consider a group of agents described by the following general heterogeneous linear dynamics with different dimensions:   
\vspace{-1pt}
\begin{align}  
& \dot{x}_{i}(t)=A_{i}x_{i}(t)+B_{i}u_{i}(t),  \  y_{i}(t)=C_{i}x_{i}(t),  \ \text{where},  \label{simu}  \\
A_{i}&= \left[ \begin{array}{cc}
0 & 1  \\
0 & 0     
\end{array} \right],  B_{i}= \left[ \begin{array}{cc}
0 & 1   \\
1 & -2   
\end{array} \right],  C^{T}_{i}= \left[ \begin{array}{c}
1 \\  
1  
\end{array} \right], \ i=1, \notag \\ 
A_{i}&= \left[ \begin{array}{cc}
0 & -1  \\
1 & -2     
\end{array} \right],  B_{i}= \left[ \begin{array}{cc}
1 & 0   \\
3 & -1   
\end{array} \right],  C^{T}_{i}= \left[ \begin{array}{c}
-1 \\  
1  
\end{array} \right], \ i=2, \notag \\ 
A_{i}&= \left[ \begin{array}{ccc}
0 & 1 & 0  \\
0 & 0 & 1 \\
0.5 & 1 & -2  
\end{array} \right],  B_{i}= \left[ \begin{array}{cc}
1 & 0   \\
0 & 1  \\
1 & 0  
\end{array} \right],  C_{i}^{T}= \left[ \begin{array}{c}
1 \\ 
-1  \\
1  
\end{array} \right], i=3.\notag
\end{align} 

In order to cooperatively solve the original optimization problem: $ \mathrm{F}(\theta)=\sum_{i=1}^{N}f_{i}(\theta), \ i=1,2,3 $, the local objective functions for each agent are described by  $ f_{1}(\theta)=-2e^{-0.5\theta}+0.5e^{0.3\theta}$, $ f_{2}(\theta)=\theta^{4}+2\theta^{2}+2 $, and $ f_{3}(\theta)= 0.5\theta^{2}\ln(1+\theta^{2})+\theta^{2} $ \cite{Kia15AT}.

\vspace{2pt}
The random digraphs for a team of three agents are shown in Fig. \ref{Agent_Random_graphs}, where the random process is captured by the Markov chain with the state space being $ \mathcal{S}=\{1,2,3\} $, and its transition rate matrix is $ \Upsilon =\left( 
\begin{array}{ccc}
-0.1 & 0.02 & 0.08 \\ 
0.3 & -0.5 & 0.2 \\ 
0.1 & 0.1 & -0.2%
\end{array}%
\right) $, whose row summation is zero and all off-diagonal elements are nonnegative. The initial distribution is given by $[0.5882,0.1500,0.3235].$ As can be seen, each graph is disconnected, while the union of graphs contains a directed spanning tree to satisfy Assumption \ref{Assumption4}.

Next, executions of resilient distributed optimization algorithms with respect to three different cases are  presented. 
The values of the initial states  $ x_{i}(0) $, $ \varrho_{i}(0) $ and $ z_{i}(0) $ are randomly chosen in an interval $ [-10, 10] $. The controller gain matrix $ K_{i} $ is chosen so that $ A_{i}-B_{i}K_{i} $ is Hurwitz, which are provided as: $ K_{1}=[3, 5; 1.5, 1] $, $ K_{2}=[0.75, -1; 1.25, -4] $, $ K_{3}=[2.167, 1, 0.333; 0, 3, 1] $. Thus, the gain matrices $ U_{i} $,  $ W_{i} $, $ X_{i} $ (solution to (\ref{RegulationEquation})) can be determined by: $ U_{1} =[1;0.5] $, $ U_{2}=[-0.5;0] $, $ U_{3}=[-1;0] $, $ W_{1}=[1.5;0.5] $, $ W_{2}=[-0.5;-2] $, $ W_{3}=[0;-1] $, $ X_{1}=[0.5;0.5] $, $ X_{2}=[-0.5 ; \\ 0.5] $, and $ X_{3}=[0;-1;0] $, respectively.

\begin{figure}[!t]
	\centering
	\includegraphics[width=7.5cm,height=1.8cm]{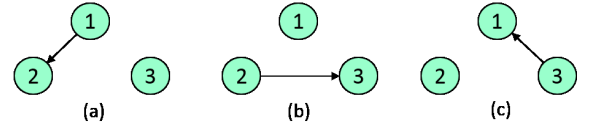}
	\caption{The communication digraphs for a team of three agents.}
	\label{Agent_Random_graphs}
\end{figure}

\begin{figure}[!t]
	\centering
	\includegraphics[width=9.0cm,height=4.7cm]{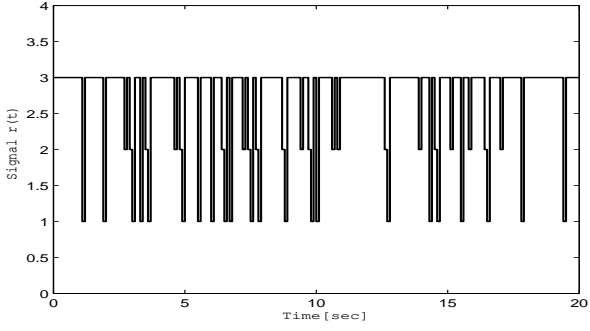}
	\caption{The signal $ r(t) $ with respect to random digraphs in Fig. \ref{Agent_Random_graphs}.}
	\label{MarkovSignal_RandomGraphs}
\end{figure}

\begin{figure}[!t]
	\centering
	\includegraphics[width=9.0cm,height=7.0cm]{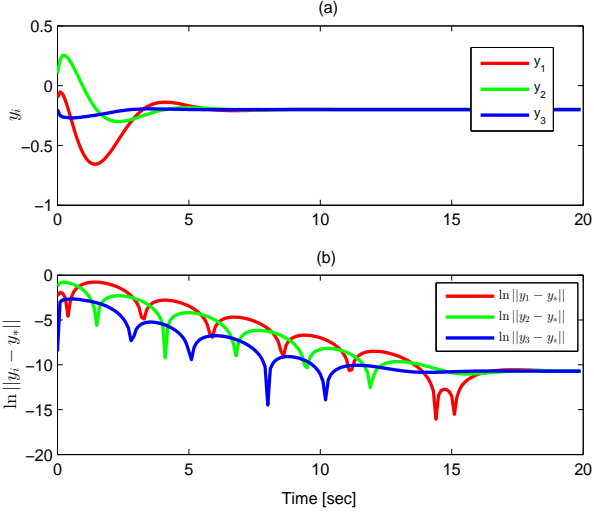}
	\caption{Execution of algorithm (\ref{Controller1}) over random digraphs under $ \alpha=2 $, $ \beta=1 $: (a) the states $y_{i}$; (b) optimal errors $ \ln ||y_{i}-y_{*}|| $, $i=1,2,3$.}
	\label{RandomGraphsOptimalStatesBeta1}
\end{figure}

\begin{figure}[!t]
	\centering
	\includegraphics[width=9.0cm,height=4.5cm]{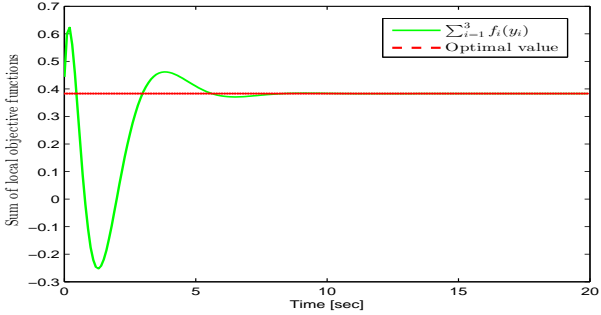}
	\caption{The sum of local objective functions $ \sum^{3}_{i=1} f_{i}(y_{i}) $.}
	\label{RandomGraphsOptimalValuesBeta1}
\end{figure}

\textit{Case 1: execution of reliable distributed optimization algorithm (\ref{Controller1}) over random digraphs and without DoS attacks}

In this simulation, the resilient distributed optimization algorithm (\ref{Controller1}) is performed under only unreliable random digraphs in the absence of DoS attacks.   
Figs. \ref{MarkovSignal_RandomGraphs}-\ref{RandomGraphsOptimalValuesBeta1} depict the simulation results for the execution of (\ref{Controller1}). The signal $ r(t) $ in Fig. \ref{MarkovSignal_RandomGraphs} describes the random process of unreliable digraphs in Fig. \ref{Agent_Random_graphs}. Fig. \ref{RandomGraphsOptimalStatesBeta1} shows the trajectories of both output states $y_{i}$ and optimal errors $ \ln ||y_{i}-y_{*}|| $, respectively, while Fig. \ref{RandomGraphsOptimalValuesBeta1} illustrates the evolution of the sum of local objective functions $ \sum^{3}_{i=1} f_{i}(y_{i}) $. It can be seen that these outputs reach consensus and converge to the optimal solution $ y_{*} $. 

To illustrate the effect of the parameter $ \beta $, Fig. \ref{RandomGraphsOptimalStates} illustrate the simulation results with different $ \beta $. As we can see, the larger value of $ \beta $ results in a faster convergence as discussed.

\begin{figure}[t!] 
	\centering
	\hspace*{-0.9em}
	\begin{tabular}{cc}	
		
		{\includegraphics[width=4.8cm,height=4.8cm]{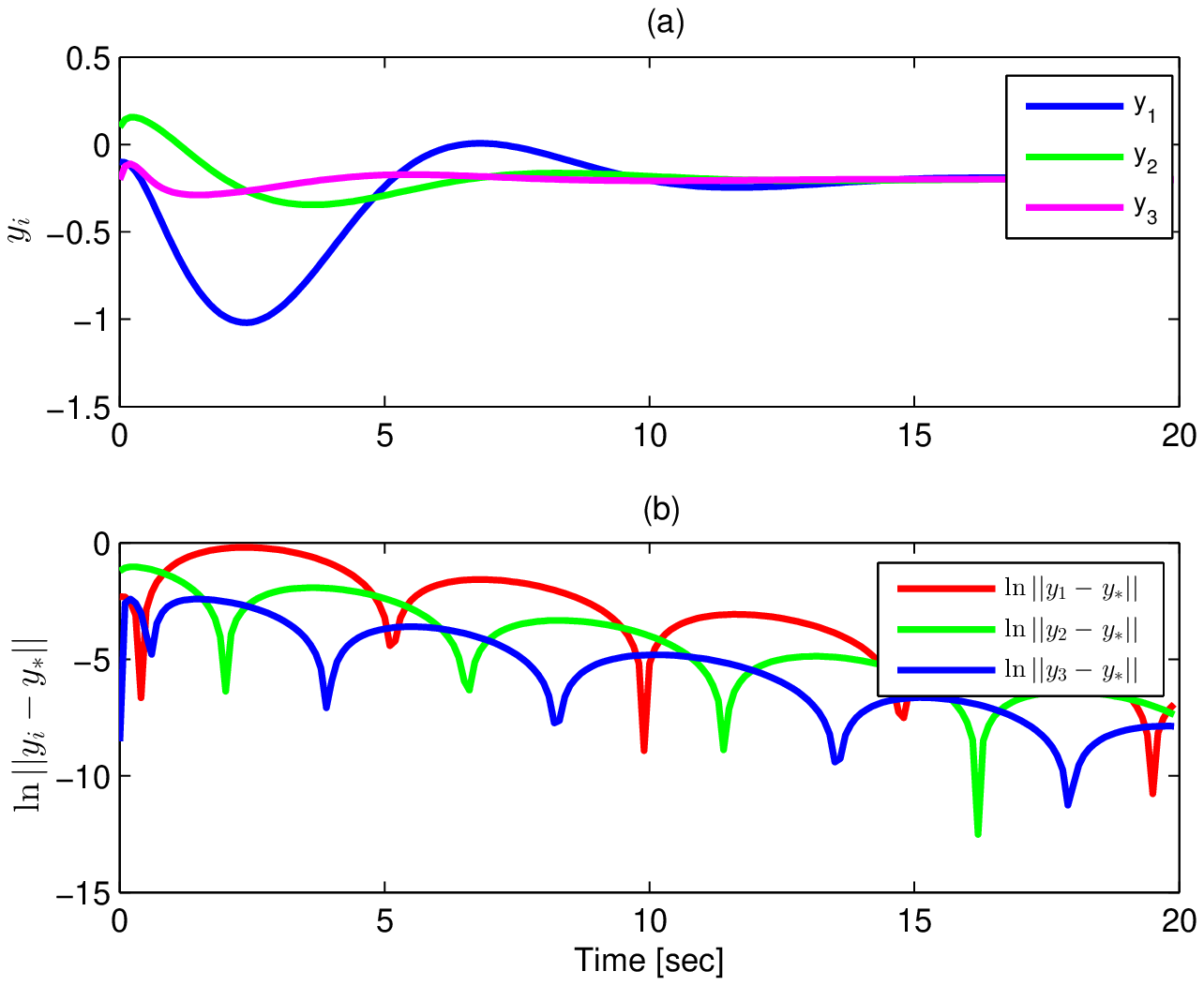}
			\label{L1}}
		
		\hspace*{-1.5em}		
		{\includegraphics[width=4.8cm,height=4.8cm]{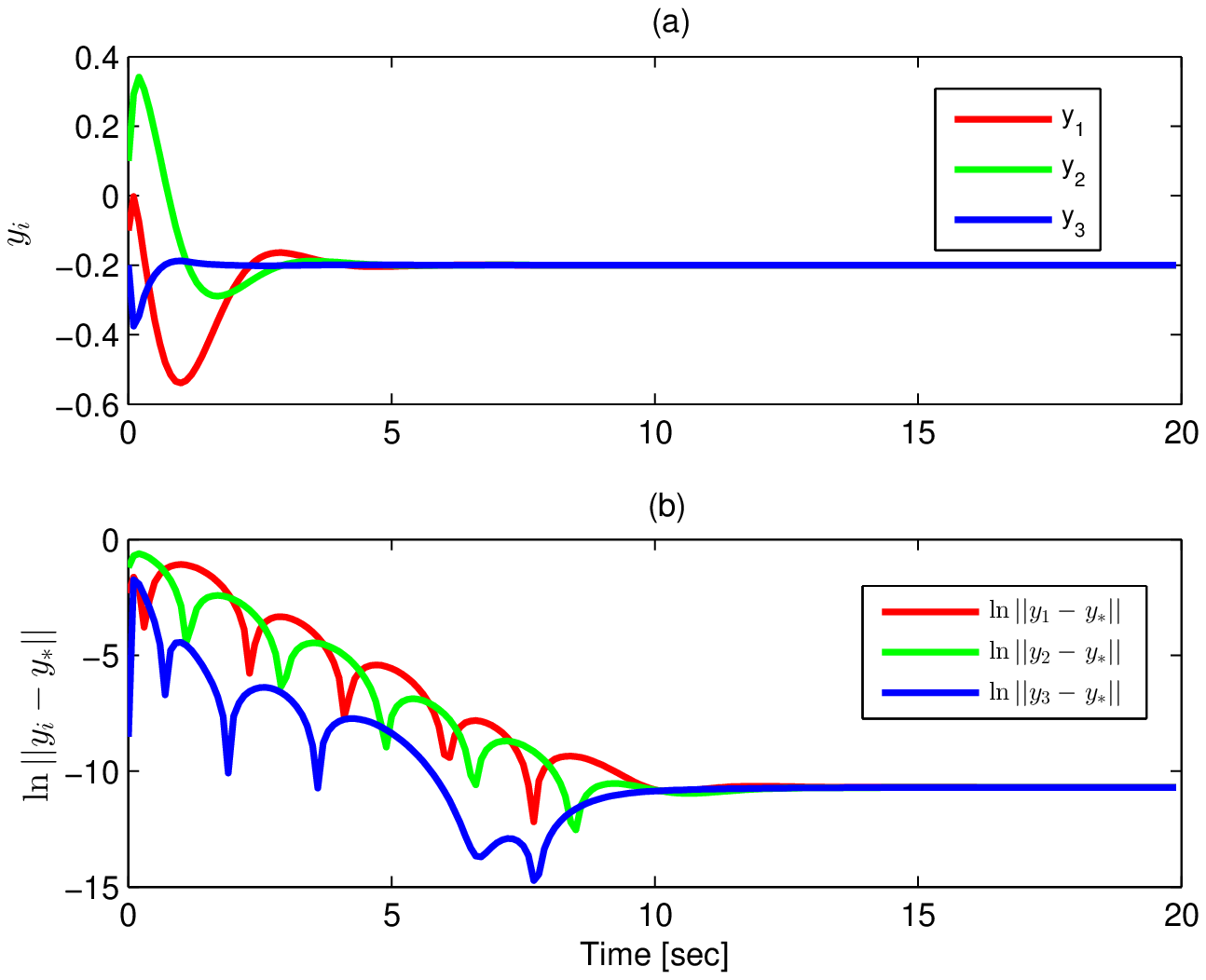}
			\label{L2}}
	\end{tabular}
	\vspace*{-3pt}
\caption{Execution of the proposed algorithm (\ref{Controller1}) over random digraphs under (a) $ \alpha=2 $, $ \beta=0.5 $; and (b) $ \alpha=2 $, $ \beta=1.5 $.}
\label{RandomGraphsOptimalStates}
\end{figure}

\textit{Case 2: execution of time-based resilient distributed optimization algorithm (\ref{Controller2}) under DoS attacks over random digraphs}

In this simulation, the resilient distributed optimization algorithm (\ref{Controller2}) with time-based communication is performed under DoS attacks over random digraphs. The simulation result is shown in Fig. \ref{ResilientOptimizationPeriodicCommunicationBoth}.  The DoS attack signal is simulated in Fig. \ref{ResilientOptimizationPeriodicCommunicationBoth}(a)  with $\tau_{a}=3s$. Based on (\ref{Condition1}) and (\ref{Condition2}), the two conditions are satisfied with $F_{a}(t_{0},t)=\frac{N_{a}(t_{0},t)}{t-t_{0}}\leq 0.01$ and $T_{a}>(\lambda_{a}+\lambda_{b})/(\lambda_{a}-\eta^{\ast })=2$. That is, the attack cannot occur more than $0.01$ times during a unit of time. Figs. \ref{ResilientOptimizationPeriodicCommunicationBoth}(b) and \ref{ResilientOptimizationPeriodicCommunicationBoth}(c) show the outputs and the errors, while Fig. \ref{ResilientOptimizationPeriodicCommunicationBoth}(d) shows the sum of local cost functions. As can be seen,  these outputs reach consensus and converge to the solution $ y_{*} $ under DoS attacks over random digraphs.

\begin{figure}[!t]
	\centering
	\hspace*{-0.9em}
	\includegraphics[width=10.2cm,height=8.2cm]{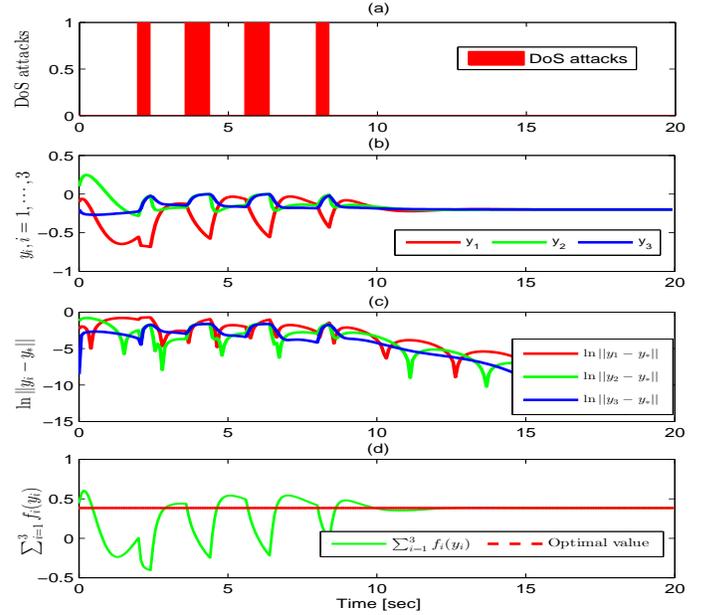}
	\caption{Execution of the proposed algorithm (\ref{Controller2}) under DoS attacks over random digraphs for $ \alpha=2 $, $ \beta=1$: (a) Dos attack signals; (b) states $y_{i}$; (c) optimal errors $ \ln ||y_{i}-y_{*}|| $, (d) sum of local objective functions $ \sum^{3}_{i=1} f_{i}(y_{i}) $, $i=1,2,3$.}
	\label{ResilientOptimizationPeriodicCommunicationBoth}
\end{figure}

\textit{Case 3: execution of event-based resilient distributed optimization algorithm (\ref{Controller3}) under DoS attacks over random digraphs}

In this part, the event-triggered resilient distributed optimization algorithm (\ref{Controller3}) has been performed under attacks over random digraphs. The simulation environments including DoS attacks and control parameters are set the same as those in Case 2. Then, the simulation results are depicted in Fig. \ref{ResilientOptimizationEventCommunication}, from which it can be seen that these outputs reach consensus and can converge to $ y_{*} $. Thus, the optimization problem  is solvable by the proposed distributed algorithm (\ref{Controller3}) with an event-based communication strategy under DoS attacks over random digraphs. Fig. \ref{Events} depicts the event time instants of agents, and there exists no Zeno-behavior.  

\begin{figure}[!t]
	\centering
	\hspace*{-0.9em}
	\includegraphics[width=10.0cm,height=8.2cm]{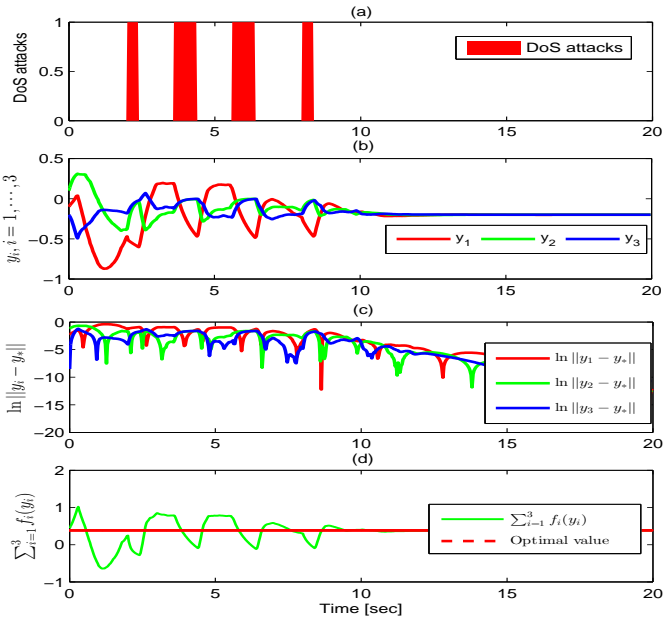}
	\vspace*{-18pt}
	\caption{Execution of the proposed algorithm (\ref{Controller3}) under DoS attacks over random digraphs for $ \alpha=2 $, $ \beta=1$: (a) Dos attack signals; (b) states $y_{i}$; (c) optimal errors $ \ln ||y_{i}-y_{*}|| $, (d) sum of local objective functions $ \sum^{3}_{i=1} f_{i}(y_{i}) $, $i=1,2,3$.}
	\label{ResilientOptimizationEventCommunication}
\end{figure}

\begin{figure}[!t]
	\centering
	\hspace*{-0.9em}
	\includegraphics[width=10.0cm,height=4.2cm]{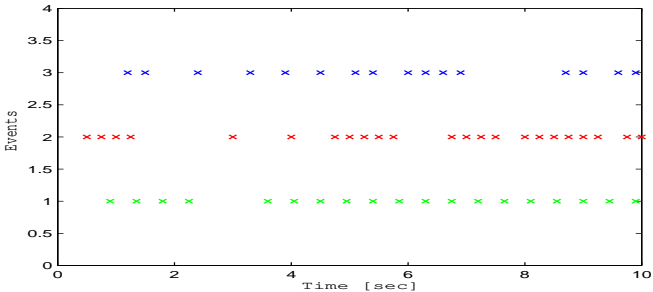}
	\vspace*{-20pt}
	\caption{The event time instants of all agents.}
	\label{Events}
\end{figure}

\vspace*{-3pt}
\section{Conclusion} 
\vspace*{-2pt}
In this paper, resilient exponential distributed convex optimization problems were studied for linear multi-agent systems under DoS attacks over random digraphs. The two types of time-based and event-based resilient distributed optimization algorithms were proposed to solve these problems, respectively. Under both algorithms, the global minimizer was achieved exponentially, provided that an explicit analysis of the frequency and duration of attacks was established. In addition, it was proved that there were no Zeno behavior occurring under the dynamic event-triggered condition. Future work will investigate resilient constrained optimization.  
  

\end{document}